\documentclass[11pt,dvipsnames,letterpaper]{amsart}

\usepackage{style}
\usepackage[left=3.35cm,right=3.35cm,bottom=3.5cm,top=3.4cm]{geometry}

\usepackage{times}

\def\BarthPeters{\(\mathrm{BP}\)}

\tikzstyle{nodal}=[circle,draw,fill=black,inner sep=0pt, minimum width=4pt]
\tikzstyle{half-fiber}=[rectangle,draw=black,thick,inner sep=0pt, minimum width=5pt, minimum height=5pt]
\tikzset{double distance = 2pt}

\author{Gebhard Martin}
\address{Mathematisches Institut \\ Universität Bonn \\ Endenicher Allee 60 \\ 53115 Bonn \\ Germany}
\email{gmartin@math.uni-bonn.de} 

\author{Giacomo Mezzedimi}
\address{Institut für Algebraische Geometrie \\ Leibniz Universität Hannover \\ Welfengarten 1 \\ 30167 Hannover \\ Germany}
\email{mezzedimi@math.uni-hannover.de}

\author{Davide Cesare Veniani}
\address{Institut für Diskrete Strukturen und Symbolisches Rechnen \\ Universität Stuttgart \\ Pfaffenwaldring 57 \\ 70569 Stuttgart \\ Germany}
\email{davide.veniani@mathematik.uni-stuttgart.de}

\title{Enriques surfaces of non-degeneracy 3}

\date{\today}

\begin{document}

\begin{abstract}
We classify all non-extendable 3-sequences of half-fibers on Enriques surfaces. If the characteristic is different from 2, we prove in particular that every Enriques surface admits a 4-sequence, which implies that every Enriques surface is the minimal desingularization of an Enriques sextic, and that every Enriques surface is birational to a Castelnuovo quintic.
\end{abstract}

\maketitle

\tableofcontents

\section{Introduction}

In this paper, we continue our investigation of~\(c\)-sequences on Enriques surfaces \cite{extra-special}. Let \(X\) be an Enriques surface defined over an algebraically closed field of arbitrary characteristic \(p\).

Recall that a (non-degenerate) \emph{\(c\)-sequence} on~\(X\) is a \(c\)-tuple \((F_1,\ldots,F_c)\) of half-fibers such that \(F_i.F_j = 1 - \delta_{ij}\). We say that a \(c\)-sequence is \emph{extendable} if there is a \(c'\)-sequence with \(c' > c\) such that the former is contained in the latter, disregarding the order.
We define the \emph{non-degeneracy} \(\nd(F_1,\ldots,F_c)\) of a \(c\)-sequence \((F_1,\ldots,F_c)\) to be the maximal \(c'\) such that \((F_1,\ldots,F_c)\) extends to a \(c'\)-sequence. We define the \emph{maximal} and \emph{minimal non-degeneracy} of the surface \(X\) to be
\begin{align*}
\max \nd(X) & \coloneqq \max \nd(F_1,\ldots,F_c), \\
\min \nd(X) & \coloneqq \min \nd(F_1,\ldots,F_c),
\end{align*}
where the right-hand side runs over all \(c\)-sequences on~\(X\), for all possible \(c\). In \cite[Definition~6.1.9]{DolgachevKondoBook}, the maximal non-degeneracy of \(X\) is called \emph{non-degeneracy invariant} of \(X\) and denoted by \(\nd(X)\).

Our work is motivated by the known results on Enriques surfaces of non-degeneracy~\(1\) and~\(2\). Recall that an Enriques surface \(X\) is of \emph{type~\(\tilde{E}_8\)}, if \(X\) contains \((-2)\)-curves with the following dual graph:
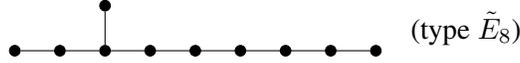
\[
\begin{tikzpicture}[scale=0.6]
    \node at (8, 0.5) {(type~\(\tilde{E}_8\))};
    \node (R4) at (0,0) [nodal] {};
    \node (R5) at (0,1) [nodal] {};
    \node (R6) at (1,0) [nodal] {};
    \node (R7) at (2,0) [nodal] {};
    \node (R8) at (3,0) [nodal] {};
    \node (R9) at (4,0) [nodal] {};
    \node (RX) at (5,0) [nodal] {};
    \node (R3) at (-1,0) [nodal] {};
    \node (R2) at (-2,0) [nodal] {};
    \node (R11) at (6,0) [nodal] {};
    \draw (R2)--(R3)--(R4) (R5)--(R4)--(RX) (RX)--(R11);
\end{tikzpicture}
\]
Then, the classification of Enriques surfaces of non-degeneracy \(1\) is as follows.

\begin{theorem}[{\cite[Theorem~3.4.1]{Cossec.Dolgachev}} or {\cite[Theorem~6.1.10]{CossecDolgachevLiedtke}}] \label{thm: non-degeneracy.1}
For an Enriques surface \(X\), the following are equivalent:
\begin{enumerate}
    \item \(\min \nd(X) = 1 \),
    \item \(\max \nd(X) = 1 \),
    \item \(X\) is of type~\(\tilde{E}_8\).
\end{enumerate}
\end{theorem}

For non-degeneracy \(2\), recall that an Enriques surface \(X\) is of \emph{type~\(\tilde{D}_8\)}, respectively of \emph{type~\(\tilde{E}_7\)}, if \(X\) contains \((-2)\)-curves with the following dual graphs:
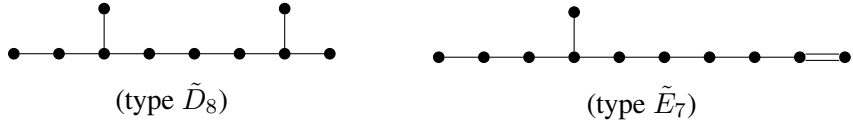
\[
\begin{tikzpicture}[scale=0.6]
    \node at (1.5,-1) {(type~\(\tilde{D}_8\))};
    \node (R4) at (0,0) [nodal] {};
    \node (R5) at (0,1) [nodal] {};
    \node (R6) at (1,0) [nodal] {};
    \node (R7) at (2,0) [nodal] {};
    \node (R8) at (3,0) [nodal] {};
    \node (R9) at (4,0) [nodal] {};
    \node (RX) at (5,0) [nodal] {};
    \node (R3) at (-1,0) [nodal] {};
    \node (R2) at (-2,0) [nodal] {};
    \node (R1) at (4,1) [nodal] {};
\draw (R2)--(R3)--(R4) (R5)--(R4)--(RX) (R1)--(R9);
\end{tikzpicture}
\qquad
\begin{tikzpicture}[scale=0.6]
    \clip (-4,-1.5) rectangle (7,1.5);
    \node at (1.5,-1) {(type~\(\tilde{E}_7\))};
    
    \node (R4) at (0,0) [nodal] {};
    \node (R5) at (0,1) [nodal] {};
    \node (R6) at (1,0) [nodal] {};
    \node (R7) at (2,0) [nodal] {};
    \node (R8) at (3,0) [nodal] {};
    \node (R9) at (4,0) [nodal] {};
    \node (RX) at (5,0) [nodal] {};
    \node (R3) at (-1,0) [nodal] {};
    \node (R2) at (-2,0) [nodal] {};
    \node (R1) at (-3,0) [nodal] {};
    \node (R11) at (6,0) [nodal] {};
    \draw (R1)--(R2)--(R3)--(R4) (R5)--(R4)--(RX);
    \draw[double] (RX)--(R11);
\end{tikzpicture}
\]
Then, the classification of Enriques surfaces of non-degeneracy \(2\) is as follows.

\begin{theorem}[{\cite[Theorem~1.3]{extra-special}}] \label{thm: non-degeneracy.2}
For an Enriques surface \(X\), the following are equivalent:
\begin{enumerate}
    \item \(\min \nd(X) = 2 \),
    \item \(\max \nd(X) = 2 \),
    \item \(X\) is of type~\(\tilde{D}_8\) or~\(\tilde{E}_7\).
\end{enumerate}
\end{theorem}
Enriques surfaces of type~\(\tilde{E}_8\), \(\tilde{D}_8\) or \(\tilde{E}_7\) only exist in characteristic~\(p = 2\), and are either classical or supersingular. They are collectively called \emph{extra-special}. Thus, \(X\) is extra-special if and only if \(\min \nd(X) \leq 2\) or, equivalently, \(\max \nd(X) \leq 2\).

In this article, we study Enriques surfaces of non-degeneracy \(3\). We say that \(X\) is of \emph{type~\(\tilde{A}_7\)}, \emph{type~\(\tilde{E}_7^{(2)}\)}, respectively of \emph{type~\(2\tilde{D}_4\)}, if \(X\) contains \((-2)\)-curves with the following dual graphs:
\[
    \begin{tikzpicture}[scale=0.6]
    \begin{scope}[scale=0.765]
    \node at (270:3) {(type~\(\tilde{A}_7\))};
\node (R1) at (180:2) [nodal] {};
\node (R2) at (135:2) [nodal] {};
\node (R3) at (90:2) [nodal] {};
\node (R4) at (45:2) [nodal] {};
\node (R5) at (0:2) [nodal] {};
\node (R6) at (315:2) [nodal] {};
\node (R7) at (270:2) [nodal] {};
\node (R8) at (225:2) [nodal] {};
\node (R9) at (intersection of R2--R7 and R3--R8) [nodal] {};
\node (R10) at (intersection of R4--R7 and R3--R6) [nodal] {};
\draw (R6)--(R7)--(R8)--(R1)--(R2)--(R3)--(R4) (R1)--(R9) (R5)--(R10) (R4)--(R5)--(R6);
    \end{scope}
    \end{tikzpicture}
    \qquad 
    \begin{tikzpicture}[scale=0.6]
    \node at (1.5,-1.5) {(type~\(\tilde{E}_7^{(2)}\))};
    \node (R1) at (0,1) [nodal] {};
    \node (R2) at (-3,0) [nodal] {};
    \node (R3) at (-2,0) [nodal] {};
    \node (R4) at (-1,0) [nodal] {};
    \node (R5) at (0,0) [nodal] {};
    \node (R6) at (1,0) [nodal] {};
    \node (R7) at (2,0) [nodal] {};
    \node (R8) at (3,0) [nodal] {};
    \node (R9) at (4,0) [nodal] {};
    \node (RX) at (5,0) [nodal] {};
    \node (R11) at (6,0) [nodal] {};
    \draw (R2)--(R3)--(R4)--(R5) (R1)--(R5)--(RX);
    \draw[double] (RX)--(R11);
    \draw (R9) to[bend left=60] (R11);
\end{tikzpicture}
    \qquad
    \begin{tikzpicture}[scale=0.6]
    \node at (270:2) {(type~\(2\tilde{D}_4\))};
\node (R0) at (-2, 0) [nodal] {};
\node (R1) at (-2,-1) [nodal] {};
\node (R2) at (-3, 0) [nodal] {};
\node (R3) at (-2, 1) [nodal] {};
\node (R4) at (-1, 0) [nodal] {};
\node (R5) at ( 0, 0) [nodal] {};
\node (R6) at ( 1, 0) [nodal] {};
\node (R7) at ( 2,-1) [nodal] {};
\node (R8) at ( 2, 1) [nodal] {};
\node (R9) at ( 3, 0) [nodal] {};
\node (R10) at (2, 0) [nodal] {};
\draw (R0)--(R4)--(R5)--(R6)--(R10) (R1)--(R0) (R2)--(R0) (R3)--(R0) (R7)--(R10) (R8)--(R10) (R9)--(R10);

\end{tikzpicture}
\]

We will see in \autoref{example:A7}, \autoref{example:E7(2)} and \autoref{example:2D4} that each of these surfaces admits a non-extendable \(3\)-sequence. 
Our main achievement is the following theorem.

\begin{theorem} \label{thm:degeneracy.3}
For an Enriques surface \(X\), the following are equivalent:
\begin{enumerate}
    \item  \(\min \nd(X) = 3 \),
    \item \(X\) is of type~\(\tilde{A}_7\), \(\tilde{E}_7^{(2)}\), or~\(2\tilde{D}_4\).
\end{enumerate}
\end{theorem}

Its proof can be obtained by combining \autoref{thm:non-extendable.special} and \autoref{thm:non-extendable.non-special}, where, in fact, we not only classify Enriques surfaces with \(\min \nd(X) = 3\), but also the structure of all non-extendable \(3\)-sequences on them. In \autoref{sec: examples}, we describe all the non-extendable \(3\)-sequences explicitly.

Enriques surfaces of type~\(\tilde{E}_7^{(2)}\) or~\(2\tilde{D}_4\) only exist in characteristic~\(p = 2\), and are either classical or supersingular. 
In contrast, Enriques surfaces of type~\(\tilde{A}_7\) exist in any characteristic (see e.g.~\cite{Martin}). In characteristic \(p = 2\), they are ordinary Enriques surfaces.
Additionally, there is always a \(4\)-sequence on surfaces of type~{\(\tilde{A}_7\)} or~\(2\tilde{D}_4\), while surfaces of type~\(\tilde{E}_7^{(2)}\) admit only three genus one fibrations.
Therefore, we deduce the following corollary.

\begin{corollary}  \label{cor:non-degeneracy.4}
For an Enriques surface \(X\), the following are equivalent:
\begin{enumerate}
    \item \(\max \nd(X) = 3\),
    \item \(X\) is of type~\(\tilde{E}_7^{(2)}\).
\end{enumerate}
In particular, if \(p \neq 2\), or if \(p = 2\) and \(X\) is ordinary, then \(\max \nd(X) \geq 4\).
\end{corollary}

Our results have deep implications on the theory of projective models of Enriques surfaces.

Historically, Enriques surfaces arose towards the end of the 19th century in the context of Castelnuovo's criterion for rationality of surfaces as examples of non-rational smooth projective surfaces whose arithmetic and geometric genus vanish. 
The original construction by Enriques gives a \(10\)-dimensional family of Enriques surfaces arising as minimal resolutions of sextic surfaces in \(\mathbb{P}^3\) which are non-normal along the edges of a tetrahedron \cite[§39]{Enriques:introduzione}. We call these sextics \emph{Enriques sextics}. 
Castelnuovo observed that every Enriques sextic can be transformed into a normal quintic with two tacnodes and two triple points \cite[§15]{Castelnuovo}, which we call \emph{Castelnuovo quintic}. 
Later, Enriques proved that every complex Enriques surface arises as a minimal resolution of a degeneration of an Enriques sextic \cite{Enriques:bigenere}. 
Modern proofs of this result were given by Artin~\cite{ArtinEnriques} and Averbukh~\cite[Chapter~X]{Shafarevich}. 
This result was extended to classical Enriques surfaces which are not of type~\(\tilde{E}_8\) in arbitrary characteristic (see \cite[Theorem~7.4]{Cossec:Projective} and \cite[Theorem~3.1]{LangEnriques1}). 

As a consequence of \autoref{thm: non-degeneracy.2}, we were able to restrict how far one has to degenerate Enriques sextics in order to obtain all Enriques surfaces as their minimal resolutions (see \cite[Corollary~1.7]{extra-special}). 
Now, as a consequence of \autoref{cor:non-degeneracy.4}, we obtain the following result, which shows that if the characteristic is different from \(2\), then no degeneration is necessary at all. 
In other words, every Enriques surface arises via Enriques' original construction and, consequently, every Enriques surface is birational to a Castelnuovo quintic.

\begin{theorem}
For \(p \neq 2\), \emph{every} Enriques surface is the minimal resolution of an Enriques sextic
\begin{equation} \label{eq: Enriques.sextic}
    x_0^2x_1^2x_2^2 + x_0^2x_1^2x_3^2 + x_0^2x_2^2x_3^2 + x_1^2x_2^2x_3^2 + x_0x_1x_2x_3Q = 0,
\end{equation}
and birationally equivalent to the corresponding Castelnuovo quintic
\begin{equation} \label{eq: Castelnuovo.quintic}
    x_0(x_1^2x_2^2 + x_1^2x_3^2 + x_2^2x_3^2 + x_0^2x_1^2) + x_1Q',
\end{equation}
where \(Q\) is a quadric and \(Q' = Q(x_2x_3,x_0x_1,x_0x_2,x_0x_3)\).
\end{theorem}
We refer the reader to \autoref{sec: projective.models} for a proof and for the corresponding statement for classical Enriques surfaces in characteristic \(2\). We leave it to the reader to formulate the analogous consequences of \autoref{cor:non-degeneracy.4} for ordinary and supersingular Enriques surfaces in characteristic \(2\).

The paper is structured as follows. In \autoref{sec: preliminaries}, we recall and prove basic results about \(c\)-sequences, with special focus on the cases \(c = 1,2,3\). The fundamental geometrical insight for the proof of \autoref{thm:degeneracy.3} is the distinction between special and non-special \(3\)-sequences (see \autoref{def: special.seq}). We study general properties of special \(3\)-sequences in \autoref{sec: special.3-seq}. Section \autoref{sec: non-ext.3-seq} on non-extendable \(3\)-sequences constitutes the technical core of this paper. Examples of all non-extendable \(3\)-sequences are provided in \autoref{sec: examples}. We then classify special and non-special non-extendable \(3\)-sequences in~\autoref{sec: special} and~\autoref{sec: non-special}, respectively.
Finally, we discuss in \autoref{sec: projective.models} the implications of our theorem on the theory of projective models of Enriques surfaces.

\begin{acknowledgments*}
We are very grateful to Dino Festi for retrieving a copy of Castelnuovo's \emph{Memorie scelte} at the Biblioteca matematica ``Giovanni Ricci'' in Milan, Italy, and to Matthias Schütt for helpful comments on a first draft.
\end{acknowledgments*}

\section{\texorpdfstring{\(c\)}{c}--sequences} \label{sec: preliminaries}

We use exactly the same notations and conventions as in our previous paper~\cite[§2]{extra-special}. We begin by recalling two results that show that the geometry of genus one fibrations and their half-fibers governs the cone of effective curves on Enriques surfaces. Linear equivalence is denoted by \(\sim\), and the Weyl group of \(X\), which is generated by reflections along classes of \((-2)\)-curves on~\(X\), by \(W_X^{\nod}\).

\begin{lemma}[{\cite[Theorem~3.2.1]{Cossec.Dolgachev}} or {\cite[Theorem~2.3.3]{CossecDolgachevLiedtke}}] \label{lem: reducibility}
If \(D\) is an effective divisor on~\(X\) with \(D^2 \geq 0\), then there exist non-negative integers \(a_i\) and \((-2)\)-curves \(R_i\) such that
\[
    D \sim D' + \sum_i a_i R_i,
\]
where \(D'\) is the unique nef divisor in the \(W_X^{\nod}\)-orbit of~\(D\). In particular, \(D^2 = D'^2\).
\end{lemma}

\begin{lemma}[{\cite[Corollary~2.2.9]{CossecDolgachevLiedtke}}] \label{lem: half-fiber}
An effective divisor \(D\) on~\(X\) is a half-fiber of a genus one fibration on~\(X\) if and only if \(D\) is nef, primitive, and \(D^2 = 0\).
\end{lemma}

A \emph{\(c\)-degenerate (canonical isotropic) \(n\)-sequence on~\(X\)} is an \(n\)-tuple of the form
\[
    \big(F_1,F_1 + R_{1,1},\hdots,F_1 + \sum_{j=1}^{r_1} R_{1,j},F_2,\ldots,F_2 + \sum_{j = 1}^{r_2}R_{2,j},\hdots, F_c, \ldots, F_{c} + \sum_{j=1}^{r_c} R_{c,j}\big),
\]
where the \(F_i\) are half-fibers of genus one fibrations on~\(X\) and the \(R_{i,j}\) are \((-2)\)-curves satisfying the following conditions:
\begin{enumerate}
    \item \(F_i.F_j = 1 - \delta_{ij}\).
    \item \(R_{i,j}.R_{i,j+1} = 1\).
    \item \(R_{i,j}.R_{k,l} = 0\) unless \((k,l) = (i,j)\) or~\((k,l)=(i,j\pm 1)\).
    \item \(F_i.R_{i,1} = 1\) and \(F_i.R_{k,l} = 0\) if \((k,l) \neq (i,1)\).
\end{enumerate}
If \(c = n\), we simply call the above sequence a \emph{\(c\)-sequence}.
We say that a \(c\)-degenerate \(n\)-sequence \emph{extends} to a \(c'\)-degenerate \(n'\)-sequence if the former is contained in the latter, disregarding the ordering.

\begin{theorem}[{\cite[Lemma~1.6.1, Theorem~3.3]{Cossec:Picard_group}}]  \label{thm: 10sequence}
If \(n \neq 9\), then every \(c\)-degenerate \(n\)-sequence on~\(X\) can be extended to a \(c'\)-degenerate \(10\)-sequence for some \(c' \geq c\).
\end{theorem}

The question whether \(c\)-sequences extend to \(c'\)-sequences with \(c' > c\) is the question treated in~\cite{extra-special} and in the present article. 

Numerical equivalence is denoted by \(\equiv\), and the numerical lattice of~\(X\) is denoted by \(\Num(X)\). Recall that \(\Num(X)\) is isomorphic to~\(E_{10}\), the unique even, unimodular, hyperbolic lattice of rank~\(10\).
By \cite[§1]{Cossec:Picard_group}, the classes of the components of a \(10\)-sequence generate a sublattice of index~\(3\) in~\(\Num(X)\). To generate the whole numerical lattice, one needs the classes~\(e_{i,j}\) described in the following lemma.

\begin{lemma}[{\cite[Lemma 1.6.2]{Cossec:Picard_group}}] \label{lem:Cossec's.lemma}
If \((f_1,\ldots,f_{10})\) is a \(10\)-tuple of vectors in \(E_{10}\) such that \(f_i.f_j = 1 - \delta_{ij}\), then for each \(i,j\) with \(1\leq i \neq j \leq 10\), there exists an isotropic vector \(e_{i,j} \in E_{10}\) such that \(e_{i,j}.f_i = e_{i,j}.f_j = 2\) and \(e_{i,j}.f_k = 1\)  for \(k \neq i,j\).
\end{lemma}

In particular, the following corollary shows that the intersection behavior of a divisor on~\(X\) with the components of a \(10\)-sequence cannot be arbitrary.
\begin{corollary} \label{cor: divisibility}
If \((f_1,\ldots,f_{10})\) is a \(10\)-tuple of vectors in \(E_{10}\) such that \(f_i.f_j = 1 - \delta_{ij}\), then for each vector \(v \in E_{10}\), the intersection number \(v. \sum f_i\) is divisible by \(3\) and if \(v \in \langle f_1,\ldots,f_{10}\rangle\), then \(v. \sum f_i\) is divisible by \(9\).
\end{corollary}
\begin{proof}
The lattice \(L = \langle f_1,\ldots,f_{10}\rangle\) has index \(3\) in \(E_{10}\) by \cite[§1]{Cossec:Picard_group}. For every element \(v \in L\), we have \(9 \mid (\sum f_i.v)\), since this holds for \(v = f_i\). For any of the \(e_{j,k}\) in \autoref{lem:Cossec's.lemma}, we have \(9 \nmid (\sum f_i.e_{j,k}) = 12\). Hence, together with the \(f_i\), any such \(e_{j,k}\) generates \(E_{10}\). Since \(3 \mid (\sum f_i.e_{j,k})\), we obtain the claim.
\end{proof}

A \(1\)-sequence is nothing but a half-fiber of a genus one fibration. We recall some basic results on genus one fibrations on Enriques surfaces. We will use Kodaira's notation for singular fibers. Singular fibers of type~\(\I_n\) are called of \emph{multiplicative} type, while all others are called of \emph{additive} type.
To state the result, we recall that, if \(p = 2\), there are three different types of Enriques surfaces, distinguished by the torsion component \(\Pic^\tau_X\) of the identity of their Picard scheme: \emph{classical} (\(\Pic^{\tau}_X(k) = \mathbb{Z}/2\mathbb{Z} = \langle \omega_X \rangle \)), \emph{ordinary} (\(\Pic^{\tau}_X = \mu_2\)), and \emph{supersingular} (\(\Pic^{\tau}_X = \alpha_2\)).

\begin{lemma}[{\cite[Theorem~4.10.3]{CossecDolgachevLiedtke}}] \label{lem: genus.1.fibrations}
Let \(f\colon X \rightarrow \IP^1\) be a genus one fibration on~\(X\).
\begin{itemize}
    \item If \(p \neq 2\), then \(f\) is an elliptic fibration with two half-fibers, and each of them is either non-singular, or singular of multiplicative type.
    \item If \(p = 2\) and \(X\) is classical, then \(f\) is an elliptic or quasi-elliptic fibration with two half-fibers, and each of them is either an ordinary elliptic curve, or singular of additive type.
    \item If \(p = 2\) and \(X\) is ordinary, then \(f\) is an elliptic fibration with one half-fiber, which is either an ordinary elliptic curve, or singular of multiplicative type.
    \item If \(p = 2\) and \(X\) is supersingular, then \(f\) is an elliptic or quasi-elliptic fibration with one half-fiber, which is either a supersingular elliptic curve, or singular of additive type. 
\end{itemize}
\end{lemma}
In particular, every genus one fibration \(f \colon X \rightarrow \IP^1\) admits a half-fiber \(F\), so it is of the form~\(|2F|\), and \((F)\) is a \(1\)-sequence.
Next, we give a bound on the number of vertical components of genus one fibrations. Recall that a genus one fibration~\(|2F|\) is called \emph{special}, if there exists a \((-2)\)-curve \(R\), called \emph{special bisection} of~\(|2F|\), with \(F.R = 1\).
\begin{proposition} \label{lem: combinatorial.0}  \label{lem: combinatorial.two.reducible.half-fibers} \label{lem: combinatorialonefibration}
If \(|2F|\) is a genus one fibration on an Enriques surface, then the following hold:
\begin{enumerate}
    \item The number of irreducible curves contained in \(s\) fibers of~\(|2F|\) is at most \(8 +s\).
    \item The sum of the root lattices associated to fibers of~\(|2F|\) embeds into \(E_8\).
\end{enumerate}
\end{proposition}

\begin{proof}
The Jacobian fibration of \(|2F|\) is a fibration of a rational surface by \cite[Proposition~4.10.1]{CossecDolgachevLiedtke}. Moreover, \(|2F|\) and its Jacobian have the same types of fibers by \cite[Theorem~4.3.19]{CossecDolgachevLiedtke}. Hence, both claims follow from the corresponding statements for rational elliptic and quasi-elliptic surfaces, which are well-known (see e.g. \cite{Oguiso.Shioda}).
\end{proof}

Next, we turn to \(2\)-sequences \((F_1,F_2)\). The linear system \(|2F_1 + 2F_2|\) associated to a \(2\)-sequence \((F_1,F_2)\) yields a morphism of degree \(2\) onto a quartic symmetroid del Pezzo surface in~\(\mathbb{P}^4\). Depending on whether \(X\) is classical or not, there are four or two half-fibers in a \(2\)-sequence and the following proposition restricts the possible combinations of half-fibers in a \(2\)-sequence.

\begin{proposition} \label{lem: F1.F2.no.common.components}  \label{prop: 2-sequence.E8} \label{lem: 2-seq.p!=2} \label{prop: G2exists}
If \((F_1,F_2)\) is a \(2\)-sequence, then the following hold:
\begin{enumerate}
    \item The half-fibers of \(|2F_1|\) and \(|2F_2|\) do not share components.
    \item The sum of the root lattices associated to the half-fibers of~\(|2F_1|\) and~\(|2F_2|\) embeds into~\(E_8\).
    \item There is a simple fiber \(G_2 \in |2F_2|\) containing all components of \(F_1\) except a simple one.
\end{enumerate}
\end{proposition}

\begin{proof}
Claim (1) is {\cite[Lemma 3.5]{DolgachevMartin}}. 

If \(X\) is classical, let \(F_1'\) and \(F_2'\) be the second half-fibers of~\(|2F_1|\) and \(|2F_2|\), respectively. To prove Claim (2), note that the sum of the root lattices associated to \(F_1,F_1',F_2,\) and \(F_2'\) embeds into \(\langle F_1,F_2 \rangle^\perp \subseteq \Num(X)\). Since \(\Num(X) \cong U \oplus E_8\) and \(\langle F_1,F_2 \rangle \cong U\), the lattice \(\langle F_1,F_2 \rangle^\perp\) is even, unimodular, and positive definite, hence isomorphic to \(E_8\). 

Claim (3) is trivial if \(F_1\) is irreducible. Suppose \(F_1 = \sum a_iR_i\) for some \((-2)\)-curves \(R_i\) and \(a_i > 0\). For each \(i\), we have \(a_iR_i.F_2 \leq F_1.F_2 = 1\) and \(R_i.F_2 \geq 0\). Hence, \(R_i.F_2 \neq 0\) for exactly one \(i\) and for this \(i\) we have \(R_i.F_2 = 1\) and \(a_i = 1\). Since \(a_i = 1\), the \((-2)\)-curve \(R_i\) is a simple component of \(F_1\) and \(F_1 - R_i\) is a connected curve with \((F_1 - R_1).F_2 = 0\), so there is a fiber \(G_2\) of \(|2F_2|\) containing \(F_1 - R_i\). By Claim (1), the fiber \(G_2\) is simple.
\end{proof}

\section{Special \texorpdfstring{\(3\)}{3}-sequences} \label{sec: special.3-seq}

Let \((F_1,F_2,F_3)\) be a \(3\)-sequence on an Enriques surface \(X\). The associated linear system \(|F_1 + F_2 + F_3|\) induces a morphism \(\varphi\colon X \to \mathbb{P}^3\). 
We recall here the properties of this linear system, referring the reader to \cite[§3.5]{CossecDolgachevLiedtke} for further details. 
Restricting \(\varphi\) to a general member of the pencil \(|F_1 + F_2|\), one checks that \(\varphi\) is birational if and only if \(|F_1 + F_2 - F_3| = \emptyset\), in which case the image of~\(\varphi\) is a non-normal sextic \(S\) in~\(\mathbb{P}^3\). If \(X\) is classical, the non-normal locus of~\(S\) is a tetrahedron of lines if and only if \(|F_1 + F_2 - F_3 + K_X| = \emptyset\).
The divisors in \(|F_1 + F_2 - F_3|\) consist of~\((-2)\)-curves, so every \(3\)-sequence on a general Enriques surface (which does not contain \((-2)\)-curves) leads to a sextic model as above. 

In this section, we provide an in-depth analysis of the so-called \emph{special} \(3\)-sequences introduced by Cossec~(cf. \cite[Definition~5.3.1]{Cossec:Picard_group}), for which one of the above conditions is violated.

\begin{definition} \label{def: special.seq}
A \(3\)-sequence \((F_1, F_2, F_3)\) is said to be \emph{special} if \(F_1 + F_2 - F_3\) is numerically equivalent to an effective divisor, i.e. \(|F_1 + F_2 - F_3| \neq \emptyset\) or \(|F_1 + F_2 - F_3 + K_X| \neq \emptyset\).
\end{definition}
This definition is actually independent of the chosen order, as the following proposition shows.

\begin{proposition}  \label{prop: propertiesof3sequences}
If \((F_1,F_2,F_3)\) is a \(3\)-sequence, then the following hold:
\begin{enumerate}
    \item We have \(|F_1 + F_2 - F_3| \neq \emptyset\) if and only if \(|F_i + F_j - F_k| \neq \emptyset\) for every permutation \((i,j,k)\) of~\((1,2,3)\).
    \item If \(|F_1 + F_2 - F_3| \neq \emptyset\) and \(X\) is classical, then \(|F_1 + F_2 - F_3 + K_X| = \emptyset\).
\end{enumerate}
\end{proposition}
\begin{proof}
For Claim (1), it suffices to show that \(|F_i + F_j - F_k| \neq \emptyset\) if \(|F_1 + F_2 - F_3| \neq \emptyset\). Let \(S_3 \in |F_1 + F_2 - F_3|\). Since \(S_3.(F_1 + F_2) = 0\) and the linear system \(|F_1 + F_2|\) is a pencil without fixed components by \cite[Proposition~2.6.1]{CossecDolgachevLiedtke}, we have \(\mathcal{O}_{S_3}(F_1 + F_2) \cong \mathcal{O}_{S_3}\). Now, consider the exact sequence
\begin{equation} \label{eq: firstexactsequence}
    0 \to \mathcal{O}_X(F_3) \to \mathcal{O}_X(F_1 + F_2) \to \mathcal{O}_{S_3} \to 0.
\end{equation}
The two sheaves on the left have $1$- resp. $2$-dimensional space of global sections, hence no higher cohomology by Riemann--Roch. Therefore, \(h^1(S,\mathcal{O}_{S_3}) = 0\) and \(h^0(S,\mathcal{O}_{S_3}) = 1\). Now, \(|2F_2|\) is a pencil without fixed components and \(S_3.F_2 = 0\), so we have the exact sequence
\[
    0 \to \mathcal{O}_X(F_2 + F_3 - F_1) \to \mathcal{O}_X(2F_2) \to \mathcal{O}_{S_3} \to 0.
\]
Since \(h^0(X,\mathcal{O}_X(2F_2)) = 2\) and \(h^0(S_3,\mathcal{O}_{S_3}) = 1\), this shows that \(|F_2 + F_3 - F_1| \neq \emptyset\). The same type of argument applies to the other permutations of~\((1,2,3)\), thus proving Claim (1).

For Claim (2), we tensor the Sequence \eqref{eq: firstexactsequence} with \(\mathcal{O}_X(K_X-F_3)\) to obtain
\[
    0 \to \mathcal{O}_X(K_X) \to \mathcal{O}_X(K_X + F_1 + F_2 - F_3) \to \mathcal{O}_{S_3}(K_X-F_3) \to 0.
\]
So, it suffices to show that \(h^0(S,\mathcal{O}_{S_3}(K_X-F_3)) = 0\). For this, consider the exact sequence
\[
    0 \to \mathcal{O}_X(K_X-F_1-F_2) \to \mathcal{O}_X(K_X- F_3) \to \mathcal{O}_{S_3}(K_X-F_3) \to 0.
\]
By Serre duality, \(H^0\) and \(H^1\) of the first two sheaves vanish, so  \(h^0(S,\mathcal{O}_{S_3}(K_X-F_3)) = 0\). 
\end{proof}

Since \(|F_i + F_j|\) is \(1\)-dimensional and without fixed components, the linear system \(|F_i + F_j - F_k|\), if non-empty, contains a unique divisor \(S_k\). In the following proposition, we collect basic properties of \(S_1,S_2,S_3\) and, in particular, we see that they are determined by their support.

Recall that for (simply laced) Dynkin diagrams, the coefficients of the highest root are given as follows \cite[Ch. VI, Prop.~25, p.~165 and pp.~250ff.]{Bourbaki}.
\[
\begin{tikzpicture}[scale=0.5, label distance=2]
    \tikzstyle{every node}=[nodal]
    \begin{scope}[xshift=-3cm]
    \node (A0) at (-1,0) [label=below:\(1\)] {};
    \node (A1) at (-2,0) [label=below:\(1\)] {};
    \node (B0) at ( 1,0) [label=below:\(1\)] {};
    \node (B1) at ( 2,0) [label=below:\(1\)] {};
    \draw (A0)--(A1) (B0)--(B1);
    \draw [dashed] (A0)--(B0);
    \end{scope}
    \begin{scope}[xshift=5cm]
    \node (A0) at (-1,0) [label=below:\(2\)] {};
    \node (A1) at (-2,0) [label=below:\(1\)] {};
    \node (B0) at ( 1,0) [label=below:\(2\)] {};
    \node (B1) at ( 2,0) [label=below:\(2\)] {};
    \node (C1) at (3,1) [label={[label distance=1]right:\(1\)}] {};
    \node (C2) at (3,-1) [label={[label distance=1]right:\(1\)}] {};
    \draw (A0)--(A1) (B0)--(B1) (C1)--(B1)--(C2);
    \draw [dashed] (A0)--(B0);
    \end{scope}
    \begin{scope}[yshift=-3cm]
    \begin{scope}[xshift=-6cm]
    \node (C0) at ( 0,0) [label=below:\(3\)] {};
    \node (A1) at (-1,0) [label=below:\(2\)] {};
    \node (A2) at (-2,0) [label=below:\(1\)] {};
    \node (B1) at ( 1,0) [label=below:\(2\)] {};
    \node (B2) at ( 2,0) [label=below:\(1\)] {};
    \node (C1) at ( 0,1) [label={[label distance=1]right:\(2\)}] {};
    \draw (C0)--(A1)--(A2) (C0)--(B1)--(B2) (C0)--(C1);
    \end{scope}
    \begin{scope}
    \node (C0) at ( 0,0) [label=below:\(4\)] {};
    \node (A1) at (-1,0) [label=below:\(3\)] {};
    \node (A2) at (-2,0) [label=below:\(2\)] {};
    \node (B1) at ( 1,0) [label=below:\(3\)] {};
    \node (B2) at ( 2,0) [label=below:\(2\)] {};
    \node (B3) at ( 3,0) [label=below:\(1\)] {};
    \node (C1) at ( 0,1) [label={[label distance=1]right:\(2\)}] {};
    \draw (C0)--(A1)--(A2) (C0)--(B1)--(B2)--(B3) (C0)--(C1);
    \end{scope}
    \begin{scope}[xshift=7cm]
    \node (C0) at ( 0,0) [label=below:\(6\)] {};
    \node (A1) at (-1,0) [label=below:\(4\)] {};
    \node (A2) at (-2,0) [label=below:\(2\)] {};
    \node (B1) at ( 1,0) [label=below:\(5\)] {};
    \node (B2) at ( 2,0) [label=below:\(4\)] {};
    \node (B3) at ( 3,0) [label=below:\(3\)] {};
    \node (B4) at ( 4,0) [label=below:\(2\)] {};
    \node (C1) at ( 0,1) [label={[label distance=1]right:\(3\)}] {};
    \draw (C0)--(A1)--(A2) (C0)--(B1)--(B2)--(B3)--(B4) (C0)--(C1);
    \end{scope}

    \end{scope}
\end{tikzpicture}
\]

\begin{proposition} \label{prop: propertiesofcharacteristiccycle}
If \(S_k \in |F_i + F_j - F_k|\), then the following hold:
\begin{enumerate}
    \item \(S_k^2 = -2\), \(S_k.F_i = S_k.F_j = 0\), and \(S_k.F_k = 2\).
    \item \(S_1.S_2 = S_1.S_3 = S_2.S_3 = 2\).
    \item There exist simple fibers \(G_1,G_2,G_3\) with \(G_l \in |2F_l|\) such that \(S_{j} + S_k = G_i\).
    \item The dual graph of the components of each \(S_k\) is a Dynkin diagram, and their coefficients are given by the corresponding highest root.
\end{enumerate}
\end{proposition}
\begin{proof}
Claims (1) and (2) are obvious, so let us prove Claims (3) and (4). Since \(S_k^2 = -2\) and \(S_k.F_i = S_k.F_j = 0\), the divisor \(S_k\) is a connected sum of~\((-2)\)-curves whose support is contained in a single fiber of~\(|2F_i|\) and \(|2F_j|\). We have \(S_{j} + S_k \sim 2F_i\), so the supports of~\(S_k\) and \(S_{j}\) are contained in the same fiber \(G_i \in |2F_i|\). 

We claim first that the support of \(S_k\) is strictly contained in the support of \(G_i\). 
If this were not the case, the support of the fiber \(G_j = S_i + S_k\) would coincide with the support of \(S_k\), since its dual graph already contains an extended Dynkin diagram.
Hence, the supports of~\(G_i\) and~\(G_j\) would coincide, and since the multiplicities of the components of \(G_i\) and~\(G_j\) are uniquely determined by the dual graph and the multiplicity of the fiber, we would have \(mG_i = nG_j\) for some \(m,n \in \{1,2\}\), a contradiction. 
Therefore, the dual graph of \(S_k\) is a connected subgraph of an extended Dynkin diagram, and, thus, it is a (simply laced) Dynkin diagram.

Now, notice that \(S_k.R \leq 0\) for any component~\(R\) of \(S_k\). 
Indeed, assume by contradiction that \(S_k.R > 0\). Then, the equations \((S_j+S_k).R = G_i.R =0\) imply that \(S_i.R = S_j.R <0\), whence \(R.G_k = R.(S_i+S_j)<0\), contradicting the fact that \(G_k\) is nef. Therefore, \(S_k\) is a multiple of the so-called fundamental cycle by \cite[Proposition~2]{Artin}, whose coefficients are given by the corresponding highest root (by uniqueness of the fundamental cycle). As \(S_k^2 = -2\), \(S_k\) must in fact coincide with the fundamental cycle.

Finally, we need to show that \(G_i\) is simple. Assume that \(G_i = 2F_i\). Then, \(S_j + S_k = 2F_i\). Since \(F_k.F_i = 1\) and \(F_k.S_k = 2\), there is a simple component \(R\) of \(F_i\) contained in \(S_k\) (with multiplicity \(2\)) but not contained in \(S_j\). Hence, \(2F_i - 2R = (S_k - 2R) + S_j\). Since \((F_i-R).R' \leq 0\) for all components \(R'\) of \(S_j\) and since \(S_j\) is the smallest effective divisor supported on \({\rm Supp}(S_j)\) with this property, we have that \((F_i-R)-S_j\) is effective. Now, the equation \((S_k-2R) = (F_i-R)+(F_i-R-S_j)\) implies that \(S_k - 2R\) contains the fundamental cycle of the support of \(F_i - R\), but this is absurd, as then \(S_k\) and \(F_i\) would have the same support. Hence, \(G_i\) is simple.
\end{proof}

The following remark explains the geometric significance of the curves \(S_k \in |F_i + F_j - F_k|\).

\begin{remark} \label{rem: triangleremark}
If \(|F_1 + F_2 - F_3| \neq \emptyset\), then the morphism \(\varphi\) induced by \(|F_1 + F_2 + F_3|\) is generically of degree \(2\) onto a normal symmetroid cubic surface \(\mathcal{C}\) in \(\mathbb{P}^3\) \cite[Theorem~3.3.3]{CossecDolgachevLiedtke}. 
More precisely:
\begin{enumerate}
    \item If \(p \neq 2\), or \(p = 2\) and \(X\) is classical, then \(\mathcal{C}\) is the \emph{Cayley cubic} given by
    \[
    x_1x_2x_3 + x_0x_2x_3 + x_0x_1x_3 + x_0x_1x_2 = 0,
    \]
    which has four nodes at the four coordinate points. There are nine lines on \(\mathcal{C}\). Six of them form a tetrahedron whose vertices are the nodes. The other three lines are \(\ell_1 = \{ x_0 + x_1 = x_2 + x_3 = 0\}, \ell_2 = \{x_0 + x_2 = x_1 + x_3 = 0\},\) and \(\ell_3 = \{x_0 + x_3 = x_1 + x_2 = 0\}\).
    If \(p \neq 2\), then the \(\ell_i\) form a triangle on~\(X\) in the plane \(\{x_0 + x_1 + x_2 + x_3 = 0\}\), while if \(p = 2\), they are coplanar concurrent lines which meet at the point \([1:1:1:1]\).
    \item If \(p = 2\) and \(X\) is ordinary, then \(\mathcal{C}\) is the cubic surface with a \(D_4^1\)-singularity given by the equation
    \[
    x_1x_2x_3 + x_0x_3^2 + x_1^2x_2 + x_1x_2^2 = 0.
    \] 
    The surface contains six lines, three of which pass through the \(D_4^1\)-singularity at \([1:0:0:0]\). The other three lines are \(\ell_1 = \{x_0 = x_1 = 0\}, \ell_2 = \{x_0 = x_2 = 0\},\) and \(\ell_3 = \{x_0 = x_1 + x_2 + x_3= 0\}\). These three lines form a triangle in the plane \(\{x_0 = 0\}\).
    \item If \(p = 2\) and \(X\) is supersingular, then \(\mathcal{C}\) is the cubic surface with a \(D_4^0\)-singularity given by the equation
    \[
    x_0x_3^2 + x_1^2x_2 + x_1x_2^2 = 0.
    \]  
     The surface contains six lines, three of which pass through the \(D_4^0\)-singularity at \([1:0:0:0]\). The other three lines are \(\ell_1 = \{x_0 = x_1 = 0\}, \ell_2 = \{x_0 = x_2 = 0\},\) and \(\ell_3 = \{x_0 = x_1 + x_2 = 0\}\). These three lines are contained in the plane \(\{x_0 = 0\}\) and meet at the single point \([0:0:0:1]\).
\end{enumerate}
By studying the long exact sequence in cohomology associated to
\[
    0 \to \mathcal{O}_X(2F_k) \to \mathcal{O}_X(F_1 + F_2 + F_3) \to \mathcal{O}_{S_k}(F_1 + F_2 + F_3) \to 0,
\]
one sees that \(h^0(S_k,\mathcal{O}_{S_k}(F_1 + F_2 + F_3)) = 2\), so \(S_k\) maps to a line via \(\varphi\). Similarly, each half-fiber of each \(|2F_i|\) maps to a line. By \cite[§3.3]{CossecDolgachevLiedtke}, the images of the half-fibers are the lines through the singular points on \(\mathcal{C}\). Since the \(S_i\) are not contained in half-fibers by \autoref{prop: propertiesofcharacteristiccycle}, they map to the lines \(\ell_i\) described above. In particular, if \(p \neq 2\), or \(p = 2\) and \(X\) is ordinary, then the divisor \(S_1 + S_2 + S_3\) maps to a triangle on \(\mathcal{C}\). 
\end{remark}

By \autoref{prop: propertiesofcharacteristiccycle}, the divisors \(S_k\) are uniquely determined by their dual graphs. 
This observation, and the geometric picture explained in \autoref{rem: triangleremark}, motivate the following definition and notation.

\begin{definition}
Given a special \(3\)-sequence \((F_1,F_2,F_3)\) and \(S_k \in |F_i + F_j - F_k|\), the dual graph of~\(S_k\) is denoted by~\(\Gamma_k\). The dual graph of \(S_1 + S_2 + S_3\) is denoted by \(\Gamma\) and is called the \emph{triangle graph} of \((F_1,F_2,F_3)\).
\end{definition}
As it turns out, there exist only finitely many possibilities for the shape of the triangle graph~\(\Gamma\), listed in \autoref{tab:S1S2S3} on page~\pageref{tab:S1S2S3}. The rest of this section is dedicated to the proof that this list is complete.

\begin{lemma} \label{tab:S1S2}
For \(G_i = S_j+S_k\), the possible types of \(G_i\), \(S_j\) and \(S_k\) are listed in the following table, up to permutation of \(S_j\) and \(S_k\).
\end{lemma}
\begin{table}[h!]
    \centering
    \begin{tabular}{ll} 
        \toprule
        \(G_i\) & \((S_j,S_k)\)  \\
        \midrule 
        \(\II^*\) & \((E_8,A_1)\) or \((E_7,D_8)\)\\
        \(\III^*\) & \((E_7,A_1)\), \((E_6,A_7)\) or \((D_6,D_6)\)\\
        \(\IV^*\) & \((E_6,A_1)\) or \((D_5,A_5)\) \\
        \(\IV\) & \((A_2,A_1)\) \\
        \(\III\) & \((A_1,A_1)\) \\
        \(\I_n^*\) & \((D_{n+4},A_1)\), \((D_k,D_{n-k+6})\), \((D_{n+3},A_3)\) or \((A_{n+3},A_{n+3})\)\\
        \(\I_n\) & \((A_k,A_{n-k})\)\\
        \bottomrule
    \end{tabular}
\end{table}
\begin{proof}
Choose a simple component \(R\) of \(G_i\). Then, without loss of generality, we may assume that \(R\) is in \(S_k\) and not in \(S_j\). Fix the type of \(G_i\) and \(S_k\). A combination \(G_i,S_j,S_k\) occurs in the table if and only if \(G_i - S_k\) is the fundamental cycle of its support. From here, compiling the list in the lemma is straightforward.
\end{proof}

\begin{proposition} \label{prop:S1S2S3}
Up to permutation, the only possibilities for the types of \((S_1,S_2,S_3)\) and their dual graphs are listed in \autoref{tab:S1S2S3}.
\end{proposition}
\begin{proof}
We use \autoref{tab:S1S2} throughout. Once we have determined the types of \(S_i\) and \(G_i\), drawing the dual graph is straightforward.
 
If \(S_1\) is of type~\(E_8\), then both \(S_2\) and \(S_3\) must be of type~\(A_1\). Here, \(G_2\) and \(G_3\) are of type~\(\II^*\), and \(G_1\) is of type~\(\I_2\) or \(\III\).
 
If \(S_1\) is of type~\(E_7\), then \(S_2\) and \(S_3\) are of type~\(D_8\) or \(A_1\). 
Since \(X\) admits no fiber of type~\(\I_n^*\) with \(n \geq 5\) by \autoref{lem: combinatorial.0}, \(S_2\) and \(S_3\) cannot be both of type~\(D_8\), so we may assume that \(S_3\) is of type~\(A_1\). If \(S_2\) is of type~\(D_8\), then the type of \((G_1,G_2,G_3)\) is \((\I_4^*,\III^*,\II^*)\). If \(S_2\) is of type~\(A_1\), then \(G_2\) and \(G_3\) are of type~\(\III^*\), and \(G_1\) is of type~\(\I_2\) or \(\III\).
 
If \(S_1\) is of type~\(E_6\), then \(S_2\) and \(S_3\) are of type~\(A_7\) or \(A_1\). 
If both are of type~\(A_7\), then the type of \((G_1,G_2,G_3)\) is \((\I_4^*,\III^*,\III^*)\), because $X$ admits no fiber of type \({\rm I}_{14}\) by \autoref{lem: combinatorial.0}.
If \(S_2\) is of type~\(A_7\) and \(S_3\) is of type~\(A_1\), then the type of \((G_1,G_2,G_3)\) is \((\I_8,\IV^*,\III^*)\).
If \(S_2\) and \(S_3\) are of type~\(A_1\), then \(G_2\) and \(G_3\) are of type~\(\IV^*\) and \(G_1\) is of type~\(\I_2\) or \(\III\). 

Next, assume that the type of \((S_1,S_2,S_3)\) is \((D_m,D_n,D_l)\) with \(m \geq n \geq l\ge 4\). 
If two of the \(G_i\) are of type \({\rm III}^*\), then so is the third, since, by \autoref{lem: combinatorial.0}, \(X\) does not admit a fiber of type \({\rm I}_6^*\). Hence, we get \((m,n,l) = (6,6,6)\). If only one of the \(G_i\) is of type \({\rm III}^*\), then \(l = 4\), since \(X\) does not admit a fiber of type \({\rm I}_n^*\) with \(n \geq 5\). Hence, we get \((m,n,l) = (6,6,4)\).
Now, assume that all \(G_i\) are of type \({\rm I}_k^*\). If \(m \geq 5\), then all the \(S_i\) share a component, namely the leaf of the long branch of \(S_1\). Thus, there are \(l-2\) components of \(S_3\) disjoint from \(G_3\) and they span a root lattice of rank \(l-2\). The root lattice associated to \(G_3\) has rank \(m + n - 2\). By \autoref{lem: combinatorial.0}, we obtain \(m+n+l-4 \leq 8\), contradicting \(m \geq 5\). Hence, we have \((m,n,l) = (4,4,4)\) and we get two possible graphs, according to whether the \(S_i\) share a component or not. 

Now, assume that the type of \((S_1,S_2,S_3)\) is \((D_m,D_n,A_l)\), with \(m \geq n\ge 4\). First, we claim that \(G_1\) and \(G_2\) are of type \(\I_k^*\) for a suitable \(k\).  If one of them is not, then \(l = 5\), \(G_1\) and \(G_2\) are of type \(\IV^*\) and \(G_3\) is of type \(\I_4^*\). But then the central vertex of \(S_3\) is the vertex of valency \(3\) in \(S_1\) and \(S_2\), which is absurd.
Next, if \(G_3\) is of type~\(\III^*\), then \(l \in \{1,3\}\). 
If \(l = 3\), then the central vertex of \(S_3\) would have to be the simple vertex on the long tail of~\(S_1\) and~\(S_2\). 
But these are two distinct curves on~\(X\), since \(S_1 + S_2 = G_3\) is of type~\(\III^*\), so this case does not occur. Thus, \(l = 1\) and the type of \((G_1,G_2,G_3)\) is \((\I_2^*,\I_2^*,\III^*)\).
Finally, if \(G_3\) is of type~\(\I_{m+n-6}^*\), then both \(l =3\) and \(l=1\) are possible and the type of \((G_1,G_2,G_3)\) is \((\I_{n-3}^*,\I_{m-3}^*,\I_{m+n-6}^*)\) in the former case and \((\I_{n-4}^*,\I_{m-4}^*,\I_{m+n-6}^*)\) in the latter.

Next, assume that the type of \((S_1,S_2,S_3)\) is \((D_m,A_n,A_l)\), with \(m\ge 4\). We assume \(n \geq l\). We have \(n,l \in \{1,3,5\}\). 
If \(n = 5\), then necessarily \(m = 5\) and \(l\in \{1,3,5\}\). The type of \((G_1,G_2,G_3)\) is \((\I_6,\I_1^*,\IV^*), (\I_8,\I_2^*,\IV^*),\) and \((\I_2^*,\IV^*,\IV^*)\), respectively.
Moreover, the cases where \(n,l\in \{1,3\}\) are all possible and, up to the ambiguity between \(\I_2\) and \(\III\), the type of \((S_1,S_2,S_3)\) determines the type of \((G_1,G_2,G_3)\).
 
Finally, assume that the type of \((S_1,S_2,S_3)\) is \((A_m,A_n,A_l)\). We assume \(m\ge n \ge l\). If the~\(G_i\) are all of type~\(\I_k^*\), then \(m=n=l \ge 3\). If \(G_3\) is of type~\(\I_k^*\), then \(m=n\ge 3\). Notice that, if \(l \neq 1\), then there are two different ways in which the cycle \(S_3\) can meet the simple components of \(G_3\).
Finally, if the \(G_i\) are all of type~\(\I_k\) (or \(\III\) or \(\IV\)), then \(m,n,l\ge 1\). Notice again that, if \(l \neq 1\), then there are two different ways in which the cycle \(S_3\) can meet \(G_3\).
\end{proof}

\begin{remark} \label{rk: lessthan11}
The triangle graph \(\Gamma\) has at most \(11\) components. Indeed, we have \(G_1.S_1=2\), so there are at most \(2\) components in \(S_1\) that are not orthogonal to \(F_1\). The other \(k\) components of the triangle graph lie in fibers of \(|2F_1|\), and \(G_1\) is the only whole fiber of \(|2F_1|\) contained in the triangle graph, so \(k\le 9\) by \autoref{lem: combinatorial.0}. Similarly, if the triangle graph contains a unique component not lying in fibers of \(|2F_1|\), then \(|\Gamma|\le 10\). From this, we deduce the upper bounds on \(m,n,\) and \(l\) in \autoref{tab:S1S2S3}.
\end{remark}

\begin{remark}
By \autoref{rem: triangleremark}, the divisors \(S_i\) map to lines on the cubic surface which is the image of the morphism induced by \(|F_1 + F_2 + F_3|\). If there exists a curve simultaneously contained in \(S_1,S_2,\) and \(S_3\), then the images of the \(S_i\) have to meet in a single point. By \autoref{rem: triangleremark}, this implies that \(p = 2\) and \(X\) is classical or supersingular. The interested reader can use this criterion to compile a list of possible triangle graphs if \(p \neq 2\) or \(X\) is ordinary from the list in \autoref{tab:S1S2S3}.
\end{remark}

\begin{remark} \label{rem: prop31inextraspecial}
In \cite[Proposition~3.1]{extra-special}, we proved that every \(2\)-sequence \((F_1,F_2)\) such that there exist simple fibers \(G_1 \in |2F_1|\) and \(G_2 \in |2F_2|\) that share a component extends to a \(3\)-sequence \((F_1,F_2,F_3)\). The half-fiber \(F_3\) we constructed satisfies the property that there exists a configuration of \((-2)\)-curves \(\sum a_iR_i\) such that \(F_3 + \sum a_iR_i \in |F_1 + F_2|\). In other words, \(F_1 + F_2 - F_3 \equiv \sum a_i R_i\), so \((F_1,F_2,F_3)\) is special. 
Therefore, a \(2\)-sequence \((F_1,F_2)\) extends to a special \(3\)-sequence if and only if there are two simple fibers \(G_1 \in |2F_1|, G_2 \in |2F_2|\) that share a component. (For the converse, \(S_3\) provides such a component.)
\end{remark}

Finally, we note that most Enriques surfaces that contain a \((-2)\)-curve contain a triangle graph.

\begin{theorem}
If \(X\) contains a \((-2)\)-curve and is not extra-special, then \(X\) admits a special \(3\)-sequence. In particular, \(X\) is a double cover of one of the cubic surfaces in \autoref{rem: triangleremark} and the dual graph of \(X\) contains one of the triangle graphs listed in \autoref{tab:S1S2S3}.
\end{theorem}

\begin{proof}
Let \(F_1\) be a half-fiber on \(X\). By \autoref{thm: non-degeneracy.2}, we can extend \(F_1\) to a \(3\)-sequence \((F_1,F_2,F_3)\) and further to a \(c\)-degenerate \(10\)-sequence for some \(c \geq 3\). 

If \(c \neq 10\), then there is a \((-2)\)-curve \(R'\) in this sequence with \(R'.F_1 = R'.F_2 = R'.F_3 = 0\). By \autoref{prop: G2exists}, this implies that \(R'\) is contained in two simple fibers, so \(X\) admits a special \(3\)-sequence by \autoref{rem: prop31inextraspecial}. If \(c = 10\), we have \(R.F_i \leq G_1.F_i \leq 2\) for \(i = 2,\dots,10\) and for every component \(R\) of every fiber \(G_1\) of \(|2F_1|\). In particular, if \(G_1\) has at least three components, then one of them satisfies \(R.F_i = 0\) for at least three \(i\), so we can argue as before.

Hence, we may assume that all fibers of all genus one fibrations on \(X\) have at most \(2\) components. By \cite[Theorem A.3]{LangEnriques2}, there exists a half-fiber \(F_1\) and a \((-2)\)-curve \(R\) with \(F_1.R = 1\). Extend \((F_1,F_1+R)\) to a \(c\)-degenerate \(10\)-sequence. By our assumption, we must have \(5 \leq c \leq 9\), so we can apply the argument of the previous paragraph to \(R\) and the claim is proved.
\end{proof}

{\centering 
\setlength{\tabcolsep}{4pt}

\begin{longtable}{ccccc}
    \caption{Possible graphs \(\Gamma_i\) as subgraphs of the triangle graph \(\Gamma\) of a special \(3\)-sequence, see \autoref{prop:S1S2S3}. We use the notation \(D_3=A_3\).}
    \label{tab:S1S2S3}
  \\
    \toprule
        \((S_1,S_2,S_3)\) & \(\Gamma_1\) & \(\Gamma_2\) & \(\Gamma_3\) & \(|\Gamma|\) \\
    \midrule
  \endfirsthead

    \multicolumn{5}{c}%
    {(continued from previous page)} \\
    \toprule
        \((\Gamma_1,\Gamma_2,\Gamma_3)\) & \(\Gamma_1\) & \(\Gamma_2\) & \(\Gamma_3\) & \(|\Gamma|\) \\
    \midrule
  \endhead

    \multicolumn{5}{c}{(continued on next page)}
  \endfoot

    \bottomrule
  \endlastfoot
  
\((E_8,A_1,A_1)\) & 
    \begin{tabular}{c}
    \begin{tikzpicture}[scale=0.4]
        \node (R1) at (0,1) [nodal] {};
        \node (R3) at (-1,1) [nodal] {};
        \node (R4) at (-1,0) [nodal] {};
        \node (R5) at (0,0) [nodal] {};
        \node (R6) at (1,0) [nodal] {};
        \node (R7) at (2,0) [nodal] {};
        \node (R8) at (3,0) [nodal] {};
        \node (R9) at (4,0) [nodal] {};
        \node (RX) at (5,0) [nodal, fill=white] {};
        \node (R11) at (5,1) [nodal, fill=white] {};
        \draw [very thick] (R3)--(R4)--(R5) (R1)--(R5)--(R6)--(R7)--(R8)--(R9);
        \draw (R9)--(RX);
        \draw [double] (RX)--(R11);
        \draw [out=90,in=180] (R9) to (R11);
    \end{tikzpicture}
    \end{tabular}
    &
    \begin{tabular}{c}
    \begin{tikzpicture}[scale=0.4]
        \node (R1) at (0,1) [nodal, fill=white] {};
        \node (R3) at (-1,1) [nodal, fill=white] {};
        \node (R4) at (-1,0) [nodal, fill=white] {};
        \node (R5) at (0,0) [nodal, fill=white] {};
        \node (R6) at (1,0) [nodal, fill=white] {};
        \node (R7) at (2,0) [nodal, fill=white] {};
        \node (R8) at (3,0) [nodal, fill=white] {};
        \node (R9) at (4,0) [nodal, fill=white] {};
        \node (RX) at (5,0) [nodal] {};
        \node (R11) at (5,1) [nodal,fill=white] {};
        \draw (R3)--(R4)--(R5) (R1)--(R5)--(R6)--(R7)--(R8)--(R9);
        \draw (R9)--(RX);
        \draw [double] (RX)--(R11);
        \draw [out=90,in=180] (R9) to (R11);
    \end{tikzpicture}
    \end{tabular}
    &
    \begin{tabular}{c}
    
    \begin{tikzpicture}[scale=0.4]
        \node (R1) at (0,1) [nodal, fill=white] {};
        \node (R3) at (-1,1) [nodal, fill=white] {};
        \node (R4) at (-1,0) [nodal, fill=white] {};
        \node (R5) at (0,0) [nodal, fill=white] {};
        \node (R6) at (1,0) [nodal, fill=white] {};
        \node (R7) at (2,0) [nodal, fill=white] {};
        \node (R8) at (3,0) [nodal, fill=white] {};
        \node (R9) at (4,0) [nodal, fill=white] {};
        \node (RX) at (5,0) [nodal, fill=white] {};
        \node (R11) at (5,1) [nodal] {};
        \draw (R3)--(R4)--(R5) (R1)--(R5)--(R6)--(R7)--(R8)--(R9);
        \draw (R9)--(RX);
        \draw [double] (RX)--(R11);
        \draw [out=90,in=180] (R9) to (R11);
    \end{tikzpicture} 
    \end{tabular}
    & \(10\) \\
\end{longtable}
 
}

\vfill

\section{Non-extendable \texorpdfstring{\(3\)}{3}-sequences} \label{sec: non-ext.3-seq}

This section is structured as follows. 
In \autoref{sec: examples}, we present four examples of non-extendable \(3\)-sequences. The surfaces on which these \(3\)-sequences occur are of type~\(\tilde{A}_7,\tilde{E}_7^{(2)},\) and \(2\tilde{D}_4\). 
In \autoref{sec: special}, we prove \autoref{thm:non-extendable.special}, which shows that the first three examples are in fact the only special non-extendable \(3\)-sequences that can occur on an Enriques surface. 
In \autoref{sec: non-special}, we prove \autoref{thm:non-extendable.non-special}, which shows that the fourth example is the only type of non-special non-extendable \(3\)-sequence. 

Taken together, the results of this section give a complete classification of all non-extendable \(3\)-sequences and of all Enriques surfaces of non-degeneracy \(3\).

\begin{notation*}
Throughout this section, thick lines and dashed lines in the dual graphs indicate simple fibers and half-fibers, respectively.
\end{notation*}

\subsection{Examples} \label{sec: examples}
We will now present the four types of non-extendable \(3\)-sequences.
\begin{example} \label{example:A7}
Let \(X\) be an Enriques surface of type~\(\tilde{A}_7\). Recall that this means that \(X\) contains ten \((-2)\)-curves giving rise to the following dual graph:
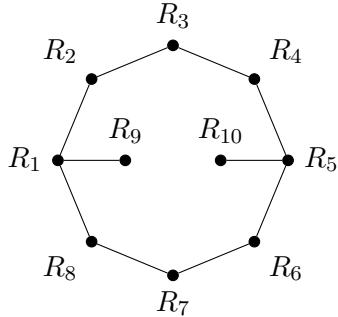
\[
    \begin{tikzpicture}[scale=0.765]
\node (R1) at (180:2) [nodal,label=left:\(R_1\)] {};
\node (R2) at (135:2) [nodal,label=above left:\(R_2\)] {};
\node (R3) at (90:2) [nodal,label=above:\(R_3\)] {};
\node (R4) at (45:2) [nodal,label=above right:\(R_4\)] {};
\node (R5) at (0:2) [nodal,label=right:\(R_5\)] {};
\node (R6) at (315:2) [nodal,label=below right:\(R_6\)] {};
\node (R7) at (270:2) [nodal,label=below:\(R_7\)] {};
\node (R8) at (225:2) [nodal,label=below left:\(R_8\)] {};
\node (R9) at (intersection of R2--R7 and R3--R8) [nodal,label=\(R_9\)] {};
\node (R10) at (intersection of R4--R7 and R3--R6) [nodal,label=\(R_{10}\)] {};

\draw (R1)--(R2)--(R3)--(R4)--(R5)--(R6)--(R7)--(R8)--(R1)--(R9) (R5)--(R10);
    \end{tikzpicture}
\]
Consider the following half-fiber \(F_0\) and fibers \(G_1,G_2,G_3\) on~\(X\).

\[
        \begin{tikzpicture}[scale=0.6]
\node (R1) at (180:2) [nodal] {};
\node (R2) at (135:2) [nodal] {};
\node (R3) at (90:2) [nodal,label=above:\(F_0\)] {};
\node (R4) at (45:2) [nodal] {};
\node (R5) at (0:2) [nodal] {};
\node (R6) at (315:2) [nodal] {};
\node (R7) at (270:2) [nodal] {};
\node (R8) at (225:2) [nodal] {};
\node (R9) at (intersection of R2--R7 and R3--R8) [nodal, fill=white] {};
\node (R10) at (intersection of R4--R7 and R3--R6) [nodal, fill=white] {};
\draw[densely dashed, very thick] (R1)--(R2)--(R3)--(R4)--(R5)--(R6)--(R7)--(R8)--(R1);
\draw (R1)--(R9) (R5)--(R10);
    \end{tikzpicture}
    \qquad
    \begin{tikzpicture}[scale=0.6]
\node (R1) at (180:2) [nodal] {};
\node (R2) at (135:2) [nodal] {};
\node (R3) at (90:2) [nodal,label=above:\(G_1\)] {};
\node (R4) at (45:2) [nodal] {};
\node (R5) at (0:2) [nodal, fill=white] {};
\node (R6) at (315:2) [nodal] {};
\node (R7) at (270:2) [nodal] {};
\node (R8) at (225:2) [nodal] {};
\node (R9) at (intersection of R2--R7 and R3--R8) [nodal] {};
\node (R10) at (intersection of R4--R7 and R3--R6) [nodal,fill=white] {};
\draw [very thick] (R6)--(R7)--(R8)--(R1)--(R2)--(R3)--(R4) (R1)--(R9);
\draw (R5)--(R10) (R4)--(R5)--(R6);
    \end{tikzpicture}
    \qquad
    \begin{tikzpicture}[scale=0.6]
\node (R1) at (180:2) [nodal] {};
\node (R2) at (135:2) [nodal] {};
\node (R3) at (90:2) [nodal,label=above:\(G_2\)] {};
\node (R4) at (45:2) [nodal] {};
\node (R5) at (0:2) [nodal] {};
\node (R6) at (315:2) [nodal] {};
\node (R7) at (270:2) [nodal, fill=white] {};
\node (R8) at (225:2) [nodal] {};
\node (R9) at (intersection of R2--R7 and R3--R8) [nodal] {};
\node (R10) at (intersection of R4--R7 and R3--R6) [nodal] {};
\draw [very thick] (R1)--(R2)--(R3)--(R4)--(R5)-- (R10)--(R5)--(R6) (R9)--(R1)--(R8); \draw (R6)--(R7)--(R8);
    \end{tikzpicture}
    \qquad
        \begin{tikzpicture}[scale=0.6]
\node (R1) at (180:2) [nodal] {};
\node (R2) at (135:2) [nodal] {};
\node (R3) at (90:2) [nodal,label=above:\(G_3\)] {};
\node (R4) at (45:2) [nodal] {};
\node (R5) at (0:2) [nodal] {};
\node (R6) at (315:2) [nodal, fill=white] {};
\node (R7) at (270:2) [nodal] {};
\node (R8) at (225:2) [nodal] {};
\node (R9) at (intersection of R2--R7 and R3--R8) [nodal] {};
\node (R10) at (intersection of R4--R7 and R3--R6) [nodal] {};
\draw [very thick] (R1)--(R2)--(R3)--(R4)--(R5)--(R10) (R7)--(R8)--(R1)--(R9) (R5)--(R10);
\draw (R5)--(R6)--(R7);
    \end{tikzpicture}
\]
By \autoref{lem: genus.1.fibrations}, the existence of \(F_0\) implies that \(p \neq 2\), or \(p = 2\) and \(X\) is ordinary. The same lemma shows that \(G_1\), \(G_2\), and \(G_3\) are simple. Choose half-fibers \(F_1,F_2,F_3\) such that \(G_i \in |2F_i|\).

\begin{claim} \label{claim: non.extendable.A7}
The half-fibers \((F_1,F_2,F_3)\) form a special non-extendable \(3\)-sequence with
triangle graph of type \((E_7,D_8,A_1)\).
\end{claim}
\begin{proof}
First, one checks that \(F_i.F_j = \frac{1}{4} G_i.G_j = 1 - \delta_{ij}\), so \((F_1,F_2,F_3)\) is indeed a \(3\)-sequence. Next, \(G_1 + G_2 - G_3 = 2R_6\), so \((F_1,F_2,F_3)\) is special. Similarly, one checks that the supports of \(G_2 + G_3 - G_1\) and \(G_1 + G_3 - G_2\) are of type \(D_8\) and \(E_7\) respectively, hence the type of the special \(3\)-sequence is \((E_7,D_8,A_1)\). It remains to prove that \((F_1,F_2,F_3)\) is non-extendable.

Seeking a contradiction, assume that there exists a half-fiber \(F_4\) with \(F_4.F_i = 1\) for \(i \in \{1,2,3\}\). 
Since \(F_4.(F_1+F_3-F_2)=1\) and \(S_2 \in |F_1+F_3-F_2|\) is supported on a configuration of \((-2)\)-curves of type \(E_7\), necessarily \(F_4\) intersects the only simple component of \(S_2\), namely \(R_4\). This, however, contradicts the facts that \(F_4.(F_2+F_3-F_1)=1\) and \(R_4\) appears with multiplicity \(2\) in \(S_1 \in |F_2+F_3-F_1|\).
\end{proof}

\begin{remark} \label{rem: max.nd.A7}
If \(X\) is of type~\(\tilde{A}_7\), then \( \max \nd(X) \geq 4\). Indeed, the following four half-fibers and simple fibers form a \(4\)-sequence:
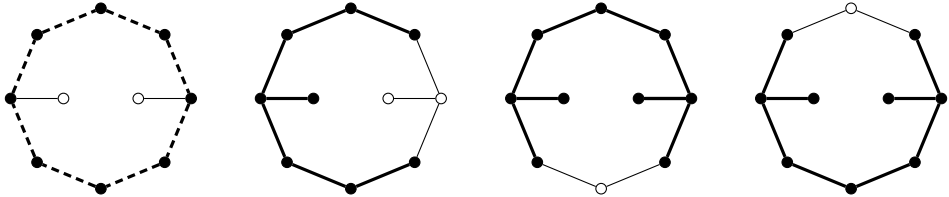
\[
    \begin{tikzpicture}[scale=0.6]
\node (R1) at (180:2) [nodal] {};
\node (R2) at (135:2) [nodal] {};
\node (R3) at (90:2) [nodal] {};
\node (R4) at (45:2) [nodal] {};
\node (R5) at (0:2) [nodal] {};
\node (R6) at (315:2) [nodal] {};
\node (R7) at (270:2) [nodal] {};
\node (R8) at (225:2) [nodal] {};
\node (R9) at (intersection of R2--R7 and R3--R8) [nodal,fill=white] {};
\node (R10) at (intersection of R4--R7 and R3--R6) [nodal,fill=white] {};
\draw [very thick, densely dashed] (R1)--(R2)--(R3)--(R4)--(R5)--(R6)--(R7)--(R8)--(R1);
\draw  (R1)--(R9) (R5)--(R10);
    \end{tikzpicture}
\qquad 
    \begin{tikzpicture}[scale=0.6]
\node (R1) at (180:2) [nodal] {};
\node (R2) at (135:2) [nodal] {};
\node (R3) at (90:2) [nodal] {};
\node (R4) at (45:2) [nodal] {};
\node (R5) at (0:2) [nodal, fill=white] {};
\node (R6) at (315:2) [nodal] {};
\node (R7) at (270:2) [nodal] {};
\node (R8) at (225:2) [nodal] {};
\node (R9) at (intersection of R2--R7 and R3--R8) [nodal] {};
\node (R10) at (intersection of R4--R7 and R3--R6) [nodal,fill=white] {};
\draw [very thick] (R6)--(R7)--(R8)--(R1)--(R2)--(R3)--(R4) (R1)--(R9);
\draw (R5)--(R10) (R4)--(R5)--(R6);
    \end{tikzpicture}
\qquad 
    \begin{tikzpicture}[scale=0.6]
\node (R1) at (180:2) [nodal] {};
\node (R2) at (135:2) [nodal] {};
\node (R3) at (90:2) [nodal] {};
\node (R4) at (45:2) [nodal] {};
\node (R5) at (0:2) [nodal] {};
\node (R6) at (315:2) [nodal] {};
\node (R7) at (270:2) [nodal, fill=white] {};
\node (R8) at (225:2) [nodal] {};
\node (R9) at (intersection of R2--R7 and R3--R8) [nodal] {};
\node (R10) at (intersection of R4--R7 and R3--R6) [nodal] {};
\draw [very thick] (R1)--(R2)--(R3)--(R4)--(R5)-- (R10)--(R5)--(R6) (R9)--(R1)--(R8); \draw (R6)--(R7)--(R8);
    \end{tikzpicture}
    \qquad
     \begin{tikzpicture}[scale=0.6]
\node (R1) at (180:2) [nodal] {};
\node (R2) at (135:2) [nodal] {};
\node (R3) at (90:2) [nodal, fill=white] {};
\node (R4) at (45:2) [nodal] {};
\node (R5) at (0:2) [nodal] {};
\node (R6) at (315:2) [nodal] {};
\node (R7) at (270:2) [nodal] {};
\node (R8) at (225:2) [nodal] {};
\node (R9) at (intersection of R2--R7 and R3--R8) [nodal] {};
\node (R10) at (intersection of R4--R7 and R3--R6) [nodal] {};
\draw [very thick] (R1)--(R8)--(R7)--(R6)--(R5) (R10)--(R5)--(R4) (R2)--(R1)--(R9);
\draw (R2)--(R3)--(R4);
    \end{tikzpicture}
\]
\end{remark}

\end{example}

\begin{example}\label{example:BP}
Let \(X\) be an Enriques surface of type~\(\tilde{A}_7\). Using the notation of \autoref{example:A7}, the curve \(R_{10}\) is a component of a fiber of \(|G_1|\). By \cite[Theorem~3.1]{Martin}, we get one of the following two graphs according to whether this fiber is a half-fiber or a simple fiber.
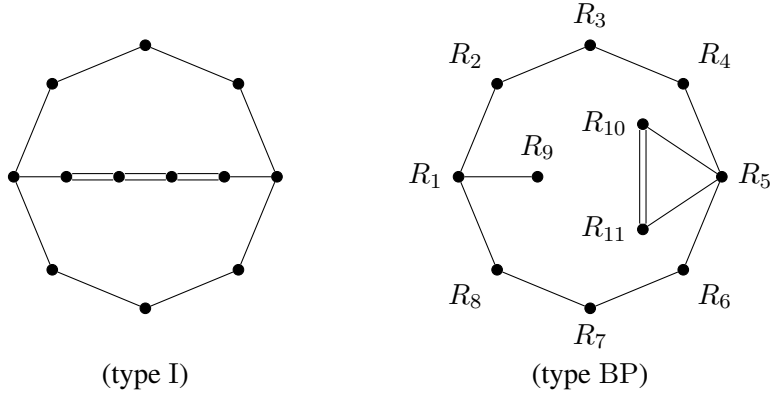
\[
\begin{tikzpicture}
    \begin{scope}[scale=0.7]
    \node at (270:3.8) {(type~\(\I\))};
\node (R1) at (180:2.5) [nodal] {};
\node (R2) at (135:2.5) [nodal] {};
\node (R3) at (90:2.5) [nodal] {};
\node (R4) at (45:2.5) [nodal] {};
\node (R5) at (0:2.5) [nodal] {};
\node (R6) at (315:2.5) [nodal] {};
\node (R7) at (270:2.5) [nodal] {};
\node (R8) at (225:2.5) [nodal] {};
\node (R9) at (-1.5,0) [nodal] {};
\node (R10) at (1.5,0) [nodal] {};
\node (R11) at (-0.5,0) [nodal] {};
\node (R12) at (0.5,0) [nodal] {};
\draw (R6)--(R7)--(R8)--(R1)--(R2)--(R3)--(R4) (R1)--(R9) (R5)--(R10) (R4)--(R5)--(R6);
\draw[double] (R9)--(R11)--(R12)--(R10);
    \end{scope}
    \end{tikzpicture}
    \qquad \qquad
    \begin{tikzpicture}
        \begin{scope}[scale=0.7]
       \node at (270:3.8) {(type~\BarthPeters)};
\node (R1) at (180:2.5) [nodal,label=left:\(R_1\)] {};
\node (R2) at (135:2.5) [nodal,label=above left:\(R_2\)] {};
\node (R3) at (90:2.5) [nodal,label=above:\(R_3\)] {};
\node (R4) at (45:2.5) [nodal,label=above right:\(R_4\)] {};
\node (R5) at (0:2.5) [nodal,label=right:\(R_5\)] {};
\node (R6) at (315:2.5) [nodal,label=below right:\(R_6\)] {};
\node (R7) at (270:2.5) [nodal,label=below:\(R_7\)] {};
\node (R8) at (225:2.5) [nodal,label=below left:\(R_8\)] {};
\node (R9) at (-1,0) [nodal,label=above:\(R_9\)] {};
\node (R10) at (1,1) [nodal,label=left:\(R_{10}\)] {};
\node (R11) at (1,-1) [nodal,label=left:\(R_{11}\)] {};
\draw (R6)--(R7)--(R8)--(R1)--(R2)--(R3)--(R4) (R1)--(R9) (R5)--(R10) (R4)--(R5)--(R6) (R5)--(R11);
\draw[double] (R11)--(R10);
\end{scope}
\end{tikzpicture}
\]

An Enriques surface containing the left configuration of \((-2)\)-curves has finite automorphism group and occurs as type~\(\I\) in Kondo's list \cite{Kondo:Enriques.finite.aut}. Enriques surfaces containing the right configuration of \((-2)\)-curves were first studied by Barth and Peters \cite{Barth.Peters} as examples of Enriques surfaces with small (but infinite) automorphism group. Now, assume that \(X\) is of type {\BarthPeters} and, overriding previous notation, consider the following three simple fibers \(G_1,G_2\) and \(G_3\):
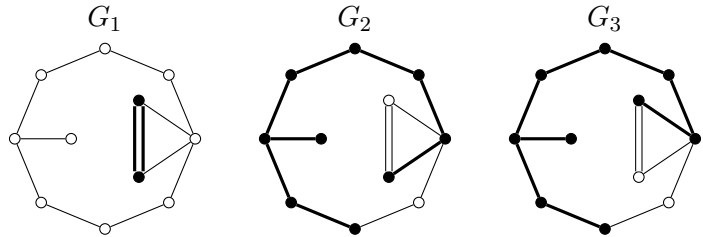
\[
    \begin{tikzpicture}[scale=0.6]
\node (R1) at (180:2) [nodal, fill=white] {};
\node (R2) at (135:2) [nodal, fill=white] {};
\node (R3) at (90:2) [nodal, fill=white, label=above:\(G_1\)] {};
\node (R4) at (45:2) [nodal, fill=white] {};
\node (R5) at (0:2) [nodal, fill=white] {};
\node (R6) at (315:2) [nodal, fill=white] {};
\node (R7) at (270:2) [nodal, fill=white] {};
\node (R8) at (225:2) [nodal, fill=white] {};
\node (R9) at (-0.75,0) [nodal, fill=white] {};
\node (R10) at (0.75,0.85) [nodal] {};
\node (R11) at (0.75,-0.85) [nodal] {};
\draw (R6)--(R7)--(R8)--(R1)--(R2)--(R3)--(R4) (R1)--(R9);
\draw (R5)--(R10) (R4)--(R5)--(R6) (R5)--(R11);
\draw[double,very thick] (R10)--(R11);
    \end{tikzpicture}
    \qquad
    \begin{tikzpicture}[scale=0.6]
\node (R1) at (180:2) [nodal] {};
\node (R2) at (135:2) [nodal] {};
\node (R3) at (90:2) [nodal, label=above:\(G_2\)] {};
\node (R4) at (45:2) [nodal] {};
\node (R5) at (0:2) [nodal] {};
\node (R6) at (315:2) [nodal, fill=white] {};
\node (R7) at (270:2) [nodal] {};
\node (R8) at (225:2) [nodal] {};
\node (R9) at (-0.75,0) [nodal] {};
\node (R10) at (0.75,0.85) [nodal,fill=white] {};
\node (R11) at (0.75,-0.85) [nodal] {};
\draw [very thick] (R1)--(R2)--(R3)--(R4)--(R5) (R7)--(R8)--(R1)--(R9)  (R5)--(R11);
\draw (R5)--(R6)--(R7) (R5)--(R10);
\draw[double] (R10)--(R11);
    \end{tikzpicture}
    \qquad
        \begin{tikzpicture}[scale=0.6]
\node (R1) at (180:2) [nodal] {};
\node (R2) at (135:2) [nodal] {};
\node (R3) at (90:2) [nodal, label=above:\(G_3\)] {};
\node (R4) at (45:2) [nodal] {};
\node (R5) at (0:2) [nodal] {};
\node (R6) at (315:2) [nodal, fill=white] {};
\node (R7) at (270:2) [nodal] {};
\node (R8) at (225:2) [nodal] {};
\node (R9) at (-0.75,0) [nodal] {};
\node (R10) at (0.75,0.85) [nodal] {};
\node (R11) at (0.75,-0.85) [nodal,fill=white] {};
\draw [very thick] (R1)--(R2)--(R3)--(R4)--(R5)--(R10) (R7)--(R8)--(R1)--(R9) ;
\draw (R5)--(R6)--(R7) (R5)--(R11);
\draw[double] (R10)--(R11);
    \end{tikzpicture}
\]
Choose half-fibers \(F_1,F_2,F_3\) such that \(G_i \in |2F_i|\).

\begin{claim}
The half-fibers \((F_1,F_2,F_3)\) form a special non-extendable \(3\)-sequence with triangle graph of type \((E_8,A_1,A_1)\).
\end{claim}
\begin{proof}
First, one checks that \(F_i.F_j = \frac{1}{4} G_i.G_j = 1 - \delta_{ij}\), so \((F_1,F_2,F_3)\) is indeed a \(3\)-sequence. Next, \(G_1 + G_2 - G_3 = 2R_{11}\), so \((F_1,F_2,F_3)\) is special. Similarly, one checks that the supports of \(G_2 + G_3 - G_1\) and \(G_1 + G_3 - G_2\) are of type \(E_8\) and \(A_1\) respectively, hence the type of the special \(3\)-sequence is \((E_8,A_1,A_1)\). It remains to prove that \((F_1,F_2,F_3)\) is non-extendable.

If \(F_4\) is a half-fiber with \(F_4.F_i = 1\) for \(i \in \{1,2,3\}\), then we have \(F_4.(F_2+F_3-F_1) = 1\). However this contradicts the fact that \(S_1 \in |F_2+F_3-F_1|\) has no simple component, since the support of \(S_1\) is a configuration of \((-2)\)-curves of type \(E_8\).
\end{proof}

\begin{remark}
An Enriques surface \(X\) of type {\BarthPeters} is in particular of type~\(\tilde A_7\), so \(\max \nd(X) \geq 4\), cf. \autoref{rem: max.nd.A7}.
\end{remark}

\end{example}

\begin{example}\label{example:E7(2)}
Let \(X\) be an Enriques surface of type~\(\tilde{E}_7^{(2)}\). Recall that this means that \(X\) contains eleven \((-2)\)-curves giving rise to the following dual graph:
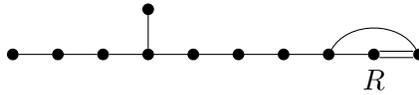
\[
\begin{tikzpicture}[scale=0.6]
    \node (R1) at (0,1) [nodal] {};
    \node (R2) at (-3,0) [nodal] {};
    \node (R3) at (-2,0) [nodal] {};
    \node (R4) at (-1,0) [nodal] {};
    \node (R5) at (0,0) [nodal] {};
    \node (R6) at (1,0) [nodal] {};
    \node (R7) at (2,0) [nodal] {};
    \node (R8) at (3,0) [nodal] {};
    \node (R9) at (4,0) [nodal] {};
    \node (RX) at (5,0) [nodal, label=below:\(R\)] {};
    \node (R11) at (6,0) [nodal] {};
    \draw (R2)--(R3)--(R4)--(R5) (R1)--(R5)--(RX);
    \draw [double] (RX)--(R11);
    \draw (R9) to[bend left=60] (R11);
\end{tikzpicture}
\]
Such a surface is a classical or supersingular Enriques surface defined over a field of characteristic~\(2\), with finite automorphism group \cite[Theorems~11.4 and 11.12]{Katsura.Kondo.Martin}. There are no other \((-2)\)-curves on~\(X\). In particular, there are only three genus one fibrations on~\(X\).

Choose fibers \(G_1,G_2,G_3\) as follows:
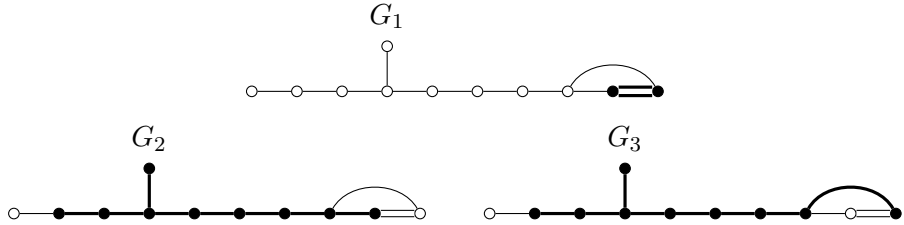
\[
\begin{tikzpicture}[scale=0.6]
    \node (R1) at (0,1) [nodal, fill=white, label=above:\(G_1\)] {};
    \node (R2) at (-3,0) [nodal, fill=white] {};
    \node (R3) at (-2,0) [nodal, fill=white] {};
    \node (R4) at (-1,0) [nodal, fill=white] {};
    \node (R5) at (0,0) [nodal, fill=white] {};
    \node (R6) at (1,0) [nodal, fill=white] {};
    \node (R7) at (2,0) [nodal, fill=white] {};
    \node (R8) at (3,0) [nodal, fill=white] {};
    \node (R9) at (4,0) [nodal, fill=white] {};
    \node (RX) at (5,0) [nodal] {};
    \node (R11) at (6,0) [nodal] {};
    \draw (R2)--(R3)--(R4)--(R5)--(R6)--(R7)--(R8) (R1)--(R5);
    \draw[double, very thick] (RX)--(R11);
    \draw (R9) to[bend left=60] (R11) (R8)--(R9)--(RX);
\end{tikzpicture}
\]\[
\begin{tikzpicture}[scale=0.6]
    \node (R1) at (0,1) [nodal, label=above:\(G_2\)] {};
    \node (R2) at (-3,0) [nodal,fill=white] {};
    \node (R3) at (-2,0) [nodal] {};
    \node (R4) at (-1,0) [nodal] {};
    \node (R5) at (0,0) [nodal] {};
    \node (R6) at (1,0) [nodal] {};
    \node (R7) at (2,0) [nodal] {};
    \node (R8) at (3,0) [nodal] {};
    \node (R9) at (4,0) [nodal] {};
    \node (RX) at (5,0) [nodal] {};
    \node (R11) at (6,0) [nodal, fill=white] {};
    \draw[very thick] (R3)--(R4)--(R5)--(R6)--(R7)--(R8)--(R9)--(RX) (R1)--(R5);
    \draw[double] (RX)--(R11);
    \draw (R2)--(R3) (R9) to[bend left=60] (R11);
\end{tikzpicture} \qquad
\begin{tikzpicture}[scale=0.6]
    \node (R1) at (0,1) [nodal, label=above:\(G_3\)] {};
    \node (R2) at (-3,0) [nodal,fill=white] {};
    \node (R3) at (-2,0) [nodal] {};
    \node (R4) at (-1,0) [nodal] {};
    \node (R5) at (0,0) [nodal] {};
    \node (R6) at (1,0) [nodal] {};
    \node (R7) at (2,0) [nodal] {};
    \node (R8) at (3,0) [nodal] {};
    \node (R9) at (4,0) [nodal] {};
    \node (RX) at (5,0) [nodal, fill=white] {};
    \node (R11) at (6,0) [nodal] {};
    \draw[very thick] (R3)--(R4)--(R5)--(R6)--(R7)--(R8)--(R9) (R1)--(R5) (R9) to[bend left=60] (R11);
    \draw[double] (RX)--(R11);
    \draw (R2)--(R3) (R9)--(RX);
\end{tikzpicture}
\]
Choose half-fibers \(F_1,F_2,F_3\) such that the corresponding fibrations have the \(G_i\) as fibers.

\begin{claim}
The half-fibers \((F_1,F_2,F_3)\) form a special non-extendable \(3\)-sequence with triangle graph of type \((E_8,A_1,A_1)\).
\end{claim}
\begin{proof}
Since \(X\) is not extra-special, and \(|2F_1|,|2F_2|,|2F_3|\) are the only genus one fibrations on~\(X\), they must form a non-extendable \(3\)-sequence. In particular, the \(G_i\) are simple fibers because \(G_i.G_j = 4\). From \(G_1 + G_2 - G_3 = 2R\), we infer that \((F_1,F_2,F_3)\) is special.
\end{proof}
\begin{remark} \label{rem: nd.E7^(2)}
If \(X\) is of type~\(\tilde{E}_7^{(2)}\), then \(\min \nd(X) = \max \nd(X) = 3\). Indeed, the surface \(X\) contains only \(3\) genus one fibrations and the corresponding half-fibers form a \(3\)-sequence.
\end{remark}

\end{example}

\begin{example}\label{example:2D4}
Let \(X\) be an Enriques surface of type~\(2\tilde{D}_4\). Recall that this means that \(X\) contains eleven \((-2)\)-curves giving rise to the following dual graph:
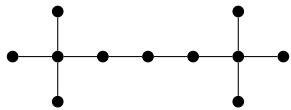
\[
    \begin{tikzpicture}[scale=0.6]
\node (R0) at (-2, 0) [nodal] {};
\node (R1) at (-2,-1) [nodal] {};
\node (R2) at (-3, 0) [nodal] {};
\node (R3) at (-2, 1) [nodal] {};
\node (R4) at (-1, 0) [nodal] {};
\node (R5) at ( 0, 0) [nodal] {};
\node (R6) at ( 1, 0) [nodal] {};
\node (R7) at ( 2,-1) [nodal] {};
\node (R8) at ( 2, 1) [nodal] {};
\node (R9) at ( 3, 0) [nodal] {};
\node (R10) at (2, 0) [nodal] {};
\draw (R0)--(R4)--(R5)--(R6)--(R10) (R1)--(R0) (R2)--(R0) (R3)--(R0) (R7)--(R10) (R8)--(R10) (R9)--(R10);
    \end{tikzpicture}
\]
Such a surface is a classical Enriques surface defined over a field of characteristic~\(2\), with finite automorphism group \cite[Theorem~13.4]{Katsura.Kondo.Martin}. There are no other \((-2)\)-curves on~\(X\).
There are exactly \(10\) genus one fibrations on~\(X\), one with two half-fibers \(F_0,F_0'\) of type~\(\I_0^*\), and the others with a (half-)fiber of type~\(\I_4^*\). 

By \autoref{thm: non-degeneracy.2}, \(F_0\) extends to a \(3\)-sequence \((F_0,F_1,F_2)\). Since all configurations of type~\(\I_4^*\) on~\(X\) have intersection number~\(2\) with \(F_0\), this means that there exist simple fibers \(G_1 \in |2F_1|\) and \(G_2 \in |2F_2|\) of type~\(\I_4^*\). The sequence \((F_0,F_1,F_2)\) is non-special by \autoref{prop: propertiesofcharacteristiccycle}, since \(|2F_0|\) contains no reducible simple fibers. Hence, we can apply \autoref{rem: prop31inextraspecial} to find a half-fiber \(F_3\) different from \(F_0\) and \(F_0'\) and such that \((F_1,F_2,F_3)\) is a special \(3\)-sequence. In particular, \(|2F_3|\) must admit a reducible simple fiber \(G_3\), necessarily of type~\(\I_4^*\). There is a unique configuration of type~\(\I_4^*\) which has intersection number \(4\) with both \(G_1\) and \(G_2\), so this is our \(G_3\), and it turns out that \((F_0,F_1,F_2,F_3)\) is in fact a \(4\)-sequence: 
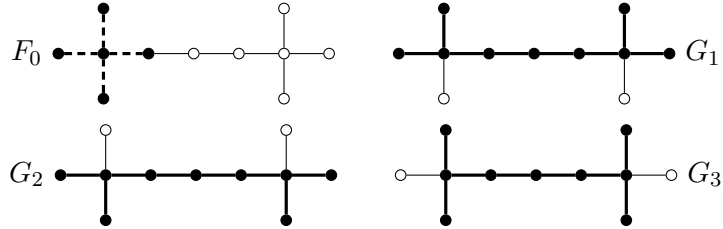
\[
    \begin{tikzpicture}[scale=0.6]
\node (R0) at (-2, 0) [nodal] {};
\node (R1) at (-2,-1) [nodal] {};
\node (R2) at (-3, 0) [nodal, label=left:\(F_0\)] {};
\node (R3) at (-2, 1) [nodal] {};
\node (R4) at (-1, 0) [nodal] {};
\node (R5) at ( 0, 0) [nodal, fill=white] {};
\node (R6) at ( 1, 0) [nodal, fill=white] {};
\node (R7) at ( 2,-1) [nodal, fill=white] {};
\node (R8) at ( 2, 1) [nodal, fill=white] {};
\node (R9) at ( 3, 0) [nodal, fill=white] {};
\node (R10) at (2, 0) [nodal, fill=white] {};
\draw [very thick,densely dashed] (R0)--(R4) (R1)--(R0) (R3)--(R0) (R2)--(R0);
\draw (R4)--(R5)--(R6) (R6)--(R10) (R7)--(R10) (R9)--(R10)  (R8)--(R10);
    \end{tikzpicture}
    \qquad
    \begin{tikzpicture}[scale=0.6]
\node (R0) at (-2, 0) [nodal] {};
\node (R1) at (-2,-1) [nodal, fill=white] {};
\node (R2) at (-3, 0) [nodal] {};
\node (R3) at (-2, 1) [nodal] {};
\node (R4) at (-1, 0) [nodal] {};
\node (R5) at ( 0, 0) [nodal] {};
\node (R6) at ( 1, 0) [nodal] {};
\node (R7) at ( 2,-1) [nodal, fill=white] {};
\node (R8) at ( 2, 1) [nodal] {};
\node (R9) at ( 3, 0) [nodal, label=right:\(G_1\)] {};
\node (R10) at (2, 0) [nodal] {};
\draw [very thick] (R0)--(R4)--(R5)--(R6)--(R10) (R2)--(R0) (R3)--(R0) (R8)--(R10) (R9)--(R10);
\draw (R1)--(R0) (R7)--(R10);
    \end{tikzpicture}
\]
\[
    \begin{tikzpicture}[scale=0.6]
\node (R0) at (-2, 0) [nodal] {};
\node (R1) at (-2,-1) [nodal] {};
\node (R2) at (-3, 0) [nodal, label=left:\(G_2\)] {};
\node (R3) at (-2, 1) [nodal, fill=white] {};
\node (R4) at (-1, 0) [nodal] {};
\node (R5) at ( 0, 0) [nodal] {};
\node (R6) at ( 1, 0) [nodal] {};
\node (R7) at ( 2,-1) [nodal] {};
\node (R8) at ( 2, 1) [nodal, fill=white] {};
\node (R9) at ( 3, 0) [nodal] {};
\node (R10) at (2, 0) [nodal] {};
\draw [very thick] (R0)--(R4)--(R5)--(R6)--(R10) (R1)--(R0) (R2)--(R0) (R7)--(R10) (R9)--(R10);
\draw (R3)--(R0) (R8)--(R10) ;
    \end{tikzpicture}
    \qquad
    \begin{tikzpicture}[scale=0.6]
\node (R0) at (-2, 0) [nodal] {};
\node (R1) at (-2,-1) [nodal] {};
\node (R2) at (-3, 0) [nodal, fill=white] {};
\node (R3) at (-2, 1) [nodal] {};
\node (R4) at (-1, 0) [nodal] {};
\node (R5) at ( 0, 0) [nodal] {};
\node (R6) at ( 1, 0) [nodal] {};
\node (R7) at ( 2,-1) [nodal] {};
\node (R8) at ( 2, 1) [nodal] {};
\node (R9) at ( 3, 0) [nodal, fill=white, label=right:\(G_3\)] {};
\node (R10) at (2, 0) [nodal] {};
\draw [very thick] (R0)--(R4)--(R5)--(R6)--(R10) (R1)--(R0) (R7)--(R10) (R3)--(R0) (R8)--(R10);
\draw (R2)--(R0) (R9)--(R10);
    \end{tikzpicture}
\]
Every other configuration of type~\(\I_4^*\) on~\(X\) has intersection number \(2\) with at least one of the \(G_i\). Hence, the following configurations are all half-fibers:
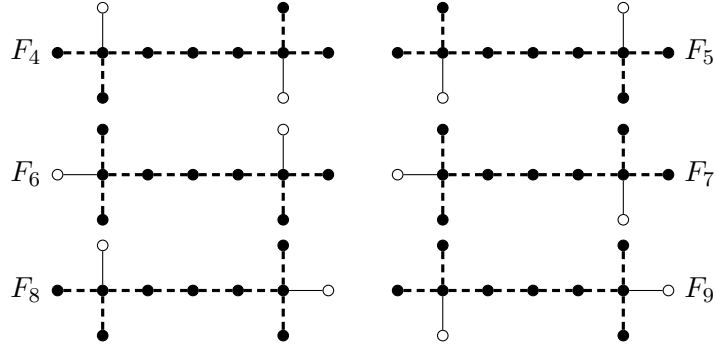
\[
   \begin{tikzpicture}[scale=0.6]
    \node (R0) at (-2, 0) [nodal] {};
\node (R1) at (-2,-1) [nodal] {};
\node (R2) at (-3, 0) [nodal, label=left:\(F_4\)] {};
\node (R3) at (-2, 1) [nodal, fill=white] {};
\node (R4) at (-1, 0) [nodal] {};
\node (R5) at ( 0, 0) [nodal] {};
\node (R6) at ( 1, 0) [nodal] {};
\node (R7) at ( 2,-1) [nodal, fill=white] {};
\node (R8) at ( 2, 1) [nodal] {};
\node (R9) at ( 3, 0) [nodal] {};
\node (R10) at (2, 0) [nodal] {};
\draw [very thick, densely dashed] (R0)--(R4)--(R5)--(R6)--(R10) (R2)--(R0) (R1)--(R0) (R9)--(R10) (R8)--(R10);
\draw (R7)--(R10)  (R3)--(R0);
    \end{tikzpicture}
    \qquad
    \begin{tikzpicture}[scale=0.6]
    \node (R0) at (-2, 0) [nodal] {};
\node (R1) at (-2,-1) [nodal, fill=white] {};
\node (R2) at (-3, 0) [nodal] {};
\node (R3) at (-2, 1) [nodal] {};
\node (R4) at (-1, 0) [nodal] {};
\node (R5) at ( 0, 0) [nodal] {};
\node (R6) at ( 1, 0) [nodal] {};
\node (R7) at ( 2,-1) [nodal] {};
\node (R8) at ( 2, 1) [nodal, fill=white] {};
\node (R9) at ( 3, 0) [nodal, label=right:\(F_5\)] {};
\node (R10) at (2, 0) [nodal] {};
\draw [very thick, densely dashed] (R0)--(R4)--(R5)--(R6)--(R10) (R2)--(R0) (R3)--(R0) (R7)--(R10) (R9)--(R10);
\draw (R1)--(R0) (R8)--(R10);
    \end{tikzpicture}
\]
\[
    \begin{tikzpicture}[scale=0.6]
    \node (R0) at (-2, 0) [nodal] {};
\node (R1) at (-2,-1) [nodal] {};
\node (R2) at (-3, 0) [nodal, fill=white, label=left:\(F_6\)] {};
\node (R3) at (-2, 1) [nodal] {};
\node (R4) at (-1, 0) [nodal] {};
\node (R5) at ( 0, 0) [nodal] {};
\node (R6) at ( 1, 0) [nodal] {};
\node (R7) at ( 2,-1) [nodal] {};
\node (R8) at ( 2, 1) [nodal, fill=white] {};
\node (R9) at ( 3, 0) [nodal] {};
\node (R10) at (2, 0) [nodal] {};
\draw [very thick, densely dashed] (R0)--(R4)--(R5)--(R6)--(R10) (R1)--(R0)  (R3)--(R0) (R7)--(R10) (R9)--(R10);
\draw (R8)--(R10) (R2)--(R0);
    \end{tikzpicture}
\qquad
    \begin{tikzpicture}[scale=0.6]
    \node (R0) at (-2, 0) [nodal] {};
\node (R1) at (-2,-1) [nodal] {};
\node (R2) at (-3, 0) [nodal, fill=white] {};
\node (R3) at (-2, 1) [nodal] {};
\node (R4) at (-1, 0) [nodal] {};
\node (R5) at ( 0, 0) [nodal] {};
\node (R6) at ( 1, 0) [nodal] {};
\node (R7) at ( 2,-1) [nodal, fill=white] {};
\node (R8) at ( 2, 1) [nodal] {};
\node (R9) at ( 3, 0) [nodal, label=right:\(F_7\)] {};
\node (R10) at (2, 0) [nodal] {};
\draw [very thick, densely dashed] (R0)--(R4)--(R5)--(R6)--(R10) (R1)--(R0) (R3)--(R0) (R8)--(R10) (R9)--(R10);
\draw (R2)--(R0)  (R7)--(R10);
\end{tikzpicture}
\]
\[
    \begin{tikzpicture}[scale=0.6]
    \node (R0) at (-2, 0) [nodal] {};
\node (R1) at (-2,-1) [nodal] {};
\node (R2) at (-3, 0) [nodal, label=left:\(F_8\)] {};
\node (R3) at (-2, 1) [nodal, fill=white] {};
\node (R4) at (-1, 0) [nodal] {};
\node (R5) at ( 0, 0) [nodal] {};
\node (R6) at ( 1, 0) [nodal] {};
\node (R7) at ( 2,-1) [nodal] {};
\node (R8) at ( 2, 1) [nodal] {};
\node (R9) at ( 3, 0) [nodal, fill=white] {};
\node (R10) at (2, 0) [nodal] {};
\draw [very thick, densely dashed] (R0)--(R4)--(R5)--(R6)--(R10) (R2)--(R0) (R1)--(R0) (R7)--(R10) (R8)--(R10);
\draw (R3)--(R0) (R9)--(R10);
    \end{tikzpicture}
    \qquad     
    \begin{tikzpicture}[scale=0.6]
    \node (R0) at (-2, 0) [nodal] {};
\node (R1) at (-2,-1) [nodal, fill=white] {};
\node (R2) at (-3, 0) [nodal] {};
\node (R3) at (-2, 1) [nodal] {};
\node (R4) at (-1, 0) [nodal] {};
\node (R5) at ( 0, 0) [nodal] {};
\node (R6) at ( 1, 0) [nodal] {};
\node (R7) at ( 2,-1) [nodal] {};
\node (R8) at ( 2, 1) [nodal] {};
\node (R9) at ( 3, 0) [nodal, fill=white, label=right:\(F_9\)] {};
\node (R10) at (2, 0) [nodal] {};
\draw [very thick, densely dashed] (R0)--(R4)--(R5)--(R6)--(R10) (R2)--(R0) (R3)--(R0) (R8)--(R10) (R7)--(R10);
\draw (R1)--(R0) (R9)--(R10);
    \end{tikzpicture}
\]

\begin{claim}
For each \(4 \leq i \leq 9\), there is a unique \(3\)-sequence containing \(F_i\). This \(3\)-sequence is non-special and non-extendable. In particular, \(\nd(F_i) = 3\).
\end{claim}
\begin{proof}
By the symmetry of the graph, it suffices to prove the statement for \(F_4\). We have \(F_4.F_j \geq 2\) for \(j \neq 1,2\) and \(F_4.F_1 = F_4.F_2 = 1\). 
Hence, \((F_1,F_2,F_4)\) is the unique \(3\)-sequence that extends~\(F_4\). Since \((F_1 + F_2 - F_4).F_5 = -2 < 0\) and \(F_5\) is nef, \((F_1,F_2,F_4)\) is non-special.
\end{proof}

\begin{remark} \label{rem: nd.2D4}
If \(X\) is of type~\(2\tilde{D}_4\), then \(\min \nd(X) = 3\) and \(\max \nd (X) = 4\). Indeed, the half-fibers \((F_0,F_1,F_2,F_3)\) form a \(4\)-sequence.
\end{remark}
\end{example}

\begin{example}
Enriques surfaces with finite automorphism group are classified into 14 types, according to the dual graph associated to the \((-2)\)-curves that they contain, some of which have already been introduced in the previous examples \cite{Katsura.Kondo.Martin, Kondo:Enriques.finite.aut, Martin}.
Every genus one fibration on an Enriques surface~\(X\) with \(|{\Aut(X)}| < \infty\) has at least one reducible fiber. Consequently, one can list all genus one fibrations on~\(X\), and compute the exact values of \(\min \nd (X)\) and \(\max \nd(X)\) (cf.~\cite{DolgachevKondoBook, Moschetti.Rota.Schaffler}). 
For the sake of completeness, these values are displayed in \autoref{tab:nd.finite.Aut(X)}, although we will never use this table in the proofs of the present paper.
\begin{table}[h]
\caption{Non-degeneracy of Enriques surfaces with finite automorphism group.}
\label{tab:nd.finite.Aut(X)}
\begin{tabular}{l cccccccc cccccc}
\toprule
 & \(\I\) & \(\II\) & \(\III\) & \(\IV\) & \(\V\) & \(\VI\) & \(\VII\) & \(\VIII\) & \(\tilde{E}_8\) & \(\tilde{D}_8\) & \(\tilde{E}_7\) & \(\tilde{E}_7^{(2)}\) & \(2\tilde{D}_4\) & \(\tilde{E}_6\) \\
 \cmidrule{2-15}
    \(\min \nd(X)\) & \(3\) & \(4\) & \(5\) & \(6\)  & \(5\) & \(7\) & \(7\) & \(4\) & \(1\) & \(2\) & \(2\) & \(3\) & \(3\) & \(4\) \\
    \(\max \nd(X)\) & \(4\) & \(7\) & \(8\) & \(10\) & \(7\) & \(10\) & \(10\) & \(7\) & \(1\) & \(2\) & \(2\) & \(3\) & \(4\) & \(4\) \\
 \bottomrule
\end{tabular}
\end{table}
\end{example}

\subsection{Special non-extendable \texorpdfstring{\(3\)}{3}-sequences} 
\label{sec: special}

In \autoref{sec: examples}, we provided three types of special non-extendable \(3\)-sequences. The aim of this section is to prove the following result, which says that our list exhausts all possible examples:
\begin{theorem} \label{thm:non-extendable.special}
If \((F_1,F_2,F_3)\) is a special non-extendable \(3\)-sequence and \(S_k \in |F_i + F_j - F_k|\), then, after possibly permuting the indices, one of the following holds:
\begin{enumerate}
    \item \((S_1,S_2,S_3)\) is of type~\((E_8,A_1,A_1)\), and \(X\) is of type \BarthPeters.
    \item \((S_1,S_2,S_3)\) is of type~\((E_8,A_1,A_1)\), and \(X\) is of type~\(\tilde{E}_7^{(2)}\).
    \item \((S_1,S_2,S_3)\) is of type~\((E_7,D_8,A_1)\), and \(X\) is of type~\(\tilde{A}_7\).
\end{enumerate}
In particular, \(X\) admits a special non-extendable \(3\)-sequence if and only if \(X\) is of type~\(\tilde{A}_7\) or~\(\tilde{E}_7^{(2)}\).
\end{theorem}

Recall that any non-extendable \(3\)-sequence \((F_1,F_2,F_3)\) is contained by \autoref{thm: 10sequence} in a \(3\)-degenerate \(10\)-sequence of the following form: 
\begin{equation} \label{eq:10-sequence}
\big(F_1, \ldots,F_1+\sum_{j=1}^{r_1} R_{1,j}, F_2, \ldots, F_2+ \sum_{j = 1}^{r_2} R_{2,j}, F_3, \ldots, F_3 + \sum_{j = 1}^{r_3} R_{3,j}\big).    
\end{equation}

\begin{proposition} \label{cor: M}
If \((F_1,F_2,F_3)\) is a special non-extendable \(3\)-sequence, then the components of the triangle graph of \((F_1,F_2,F_3)\) span a lattice of rank~\(10\) and discriminant \(1,4\) or~\(16\).
\end{proposition}

\begin{proof}
Let \(\Lambda\) be the lattice spanned by the components of the triangle graph, which we recall is the dual graph of components of the divisor \(S_1 + S_2 + S_3\).
By \autoref{thm: 10sequence}, there exists a \(3\)-degenerate \(10\)-sequence as in \eqref{eq:10-sequence}.
The lattice \(L\) spanned by the divisors in the \(10\)-sequence has index \(3\) in \(\Num(X)\).

We first show that the \(R_{i,j}\) and \(2F_i\) are contained in \(\Lambda\). We argue by induction on \(j\). Since \(S_i\) is effective and \(R_{i,1}.S_i = -1\), the curve~\(R_{i,1}\) is a component of~\(S_i\). 
Assume that \(R_{i,j}\) is contained in \(S_i\) for \(1 \leq j \leq n\). Then, \((S_i - \sum_{j=1}^n R_{i,j})\) is effective and \(R_{i,n+1}.(S_i - \sum_{j=1}^n R_{i,j}) = -1\). Hence, also \(R_{i,n+1}\) is contained in~\(S_i\).
Since \(2F_i \equiv G_i = S_j + S_k\), this shows that all the~\(R_{i,j}\) and~\(2F_i\) are in~\(\Lambda\), as claimed.

Now, we have \(F_1 + F_2 + F_3 \equiv S_1 + S_2 + S_3 \in \Lambda\), so \(\Lambda \cap L \subseteq L\) has index \(1\), \(2\) or~\(4\). Thus, it suffices to show that \(\Lambda[F_1,F_2,F_3]\) contains an element of~\(\Num(X) \setminus L\).

Without loss of generality, we may assume \(r_3 \geq 3\). 
Consider the vector \(e_{9,10} \in \Num(X)\) given by \autoref{lem:Cossec's.lemma}. Note that \(e_{9,10} \notin L\) because of \autoref{cor: divisibility}. 
By Riemann--Roch, there exists an effective divisor \(D \in \Pic(X)\), whose numerical class is \(e_{9,10}\), such that \(D.F_i = 1\) for \(i \in \{1,2,3\}\), and \(D.R_{i,j} = 0\) for all \((i,j)\) except for \(D.R_{3,r_3-1} = 1\).
By \autoref{lem: reducibility}, there exist \(D'\) and \(C\) with \(D \sim D' + C\), such that \(D'\) is nef, \(D^2 = D'^2 = 0\), and \(C\) is a linear combination of \((-2)\)-curves with non-negative coefficients.
Since \(D'.F_i \leq 1\) for all \(i\) and \((F_1,F_2,F_3)\) is non-extendable, \(D' \equiv F_i\) for some \(i \in \{1,2,3\}\). In particular, \(D.D' = 1\) and, thus, \(C^2 = -2\).

Assume \(D' \equiv F_i\) for some \(i \in \{1,2,3\}\). 
Then, \(C.F_j = 0 \) for \(j \neq i\), so \(C\) is a connected configuration of~\((-2)\)-curves contained in a fiber \(G_j\) of~\(|2F_j|\). 
Choose \(j \neq i,3\). 
We also have \(C.R_{3,r_3-1} = 1\), so \(G_j\) is the fiber containing~\(R_{3,r_3-1}\). In the second paragraph of the proof, we saw that \(R_{3,r_3-1}\) is contained in~\(S_3\), so \(G_j\) is the fiber containing~\(S_3\). By \autoref{prop: propertiesofcharacteristiccycle}, all components of \(G_j\) are in \(\Lambda\), so \(D = D' + C \in \Lambda[F_1,F_2,F_3]\), as desired.
\end{proof}

We can now prove the theorem.
 
\begin{proof}[Proof of \autoref{thm:non-extendable.special}]
Any Enriques surface of type~{\(\tilde{A}_7\)} or~\(\tilde{E}_7^{(2)}\) admits a special non-extendable \(3\)-sequence (see \autoref{example:A7} and \autoref{example:E7(2)}), so it suffices to prove the first statement.

By \autoref{cor: M}, the sublattice of \(\Num(X)\) spanned by the components of \(S_1,S_2,S_3\) has rank~\(10\) and discriminant \(d \in \{1,4,16\}\). In particular, the triangle graph~\(\Gamma\) of \((F_1,F_2,F_3)\) has at least \(10\) vertices. By inspection of \autoref{tab:S1S2S3} on page~\pageref{tab:S1S2S3}, we obtain the following \(15\) possible types of \((S_1,S_2,S_3)\) such that \(|\Gamma|\ge 10\):
\begin{align*}
& (E_8,A_1,A_1),(E_7,D_8,A_1),(E_6,A_7,A_7),(D_6,D_6,D_6),(D_6,D_6,D_4), \\
& \text{the first type of } (D_4,D_4,D_4), \\
& (D_m,D_n,A_3) \text{ with } m+n = 9,(D_m,D_n,A_1) \text{ with } m+n = 10, \\
& (D_7,A_3,A_1),(D_8,A_1,A_1),(A_m,A_m,A_m) \text{ with } 6 \leq m \leq 7, \\
& 2 \text{ types of } (A_m,A_m,A_n) \text{ with } 8 \leq m+n \leq 9, \\
& 2 \text{ types of } (A_m,A_n,A_l) \text{ with } 10 \leq m+n+l \leq 11.
\end{align*}

For seven of them, namely for \((D_4,D_4,D_4)\), \((D_m,D_n,A_3)\), \((D_m,D_n,A_1)\), \((D_m,A_3,A_1)\), \((D_m,A_1,A_1)\), the second case of type \((A_m,A_m,A_n)\) and the second case of type \((A_m,A_n,A_l)\), we have \(d > 16\), which is absurd.
We then analyze the remaining eight dual graphs. 

In case \((E_8,A_1,A_1)\), consider the genus one fibration \(|2F|\) with a fiber with two components that occurs in the graph. There is a fiber \(G \in |2F|\) of type~\(\III^*\), such that a simple component of \(G\) does not belong to \(\Gamma\). 
If \(G\) is not a simple fiber, then we obtain the graph of type~\(\tilde E_7^{(2)}\). If \(G\) is a simple fiber, we obtain the graph of type~{\BarthPeters}, and in particular \(X\) is of type~\(\tilde A_7\).

In case \((E_7,D_8,A_1)\), \(\Gamma\) is the dual graph of type~\(\tilde A_7\), so \(X\) is of type~\(\tilde A_7\). 

In case \((E_6,A_7,A_7)\), any half-fiber of type~\(\I_2^*\) in \(\Gamma\) extends the sequence. 

In case \((D_6,D_6,D_6)\), the half-fiber of type~\(\IV^*\) in \(\Gamma\) extends the sequence. 

In case \((D_6,D_6,D_4)\), the half-fiber of type~\(\I_8\) in \(\Gamma\) extends the sequence. 

In case \((A_m,A_m,A_m)\), \(6 \leq m \leq 7\), any half-fiber of type~\(\I_0^*\) in \(\Gamma\) extends the sequence. 

In the first case of type \((A_m,A_m,A_n)\), \(8 \leq m+n \leq 9\), any half-fiber of type~\(\I_4\) in \(\Gamma\) extends the sequence as soon as \(n \geq 2\). If \(n=1\), then \(m=7\) (since \(X\) cannot contain a fiber of type \(\I_5^*\) by \autoref{lem: combinatorial.0}), and it holds \(d = 64 > 16\).

In the first case of type \((A_m,A_n,A_l)\), \(10 \leq m+n+l \leq 11\), assume up to symmetry that \(m\geq n \geq l\). If \(l\geq 2\), then any half-fiber of type~\(\I_3\) in \(\Gamma\) extends the sequence. If \(l=1\), then \(m+n=9\) (since \(X\) cannot contain a fiber of type \(\I_{10}\)), and a computation shows that \(d = 144 > 16\).
\end{proof}

\subsection{Non-special non-extendable \texorpdfstring{\(3\)}{3}-sequences} \label{sec: non-special}
In \autoref{example:2D4}, we provided examples of non-special non-extendable \(3\)-sequences, all of which occurred on Enriques surfaces of type~\(2\tilde{D}_4\). The aim of this section is to prove that these are all the examples there are:

\begin{theorem} \label{thm:non-extendable.non-special}
If \((F_1,F_2,F_3)\) is a non-special non-extendable \(3\)-sequence, then \(X\) is of type~\(2\tilde{D}_4\) and, after possibly permuting the indices, both \(|2F_1|\) and \(|2F_2|\) admit a simple fiber of type~\(\I_4^*\), and \(|2F_3|\) admits a half-fiber of type~\(\I_4^*\). 
In particular, \(X\) admits a non-special non-extendable \(3\)-sequence if and only if \(X\) is of type~\(2\tilde{D}_4\).
\end{theorem}

As in the special case, any non-extendable \(3\)-sequence \((F_1,F_2,F_3)\) is contained in a \(3\)-degenerate \(10\)-sequence as in \eqref{eq:10-sequence}.

\begin{lemma} \label{lem: criterion} \label{lem: specialtycriterion}
Let \((F_1,F_2,F_3)\) be a non-special non-extendable \(3\)-sequence. If \(R\) is a \((-2)\)-curve such that \((F_1+F_2+F_3).R =1\), then \(R\) lies in a half-fiber of one of the \(|2F_i|\). In particular, all the components \(R_{i,j}\) in \eqref{eq:10-sequence} lie in half-fibers of \(|2F_1|\), \(|2F_2|\) or \(|2F_3|\).
\end{lemma}
\begin{proof}
Since the \(F_i\) are nef, we can assume \(F_1.R=F_2.R=0\) and \(F_3.R=1\) up to switching the indices. Consider the exact sequence
\[
    0 \to \mathcal{O}_X(F_1 + F_2 - R) \to \mathcal{O}_X(F_1 + F_2) \to \mathcal{O}_R(F_1 + F_2) \to 0.
\]
By \cite[Proposition~2.6.1]{CossecDolgachevLiedtke}, the linear system \(|F_1 + F_2|\) has no fixed components. Hence, by \cite[Corollary 2.6.5]{CossecDolgachevLiedtke}, we have \(h^0(X,\mathcal{O}_X(F_1 + F_2)) = 2\). On the other hand, \(\mathcal{O}_R(F_1 + F_2) = \mathcal{O}_R\), since \((F_1 + F_2).R = 0\). Thus, \(h^0(R,\mathcal{O}_R(F_1 + F_2)) = 1\) and the long exact sequence in cohomology shows that \(|F_1 + F_2 - R|\) is non-empty.

Let \(D \in |F_1 + F_2 - R|\). Then, \(D^2 = 0\) and \(D.F_1 = D.F_2 = D.F_3= 1\). By \autoref{lem: reducibility}, we can write \(D \sim D' + \sum_{i} a_iR_i\) with \(D' \neq 0\) nef, \(D'^2 = 0\), and the \(R_i\) \((-2)\)-curves. 
Since the \(F_i\) are nef, we have \(D'.F_i \le 1\) for \(i=1,2,3\), so \(D'\) is a half fiber \(F\) on~\(X\) by \cite[Corollary 2.2.9]{CossecDolgachevLiedtke}. 
As \((F_1,F_2,F_3)\) is non-extendable, \(D'\) must be a half-fiber of \(|2F_1|\), \(|2F_2|\) or \(|2F_3|\).

Assume first that \(F \equiv F_3\). 
Then, \(R + \sum a_iR_i \equiv F_1 + F_2 - F_3\), contradicting the assumption that \((F_1,F_2,F_3)\) is non-special.

Assume instead that \(F \equiv F_1\) or \(F \equiv F_2\). After a suitable permutation of these half-fibers, we may assume \(F = F_1\). Then, \(|D - F| = |F_2 - R| = |\sum_{i} a_iR_i|\), so \(F_2 = R + \sum_{i} a_i R_i\) and \(R\) belongs to the half-fiber \(F_2\).

For the second part of the statement, note that \(R_{1,j}\) is in the same fiber of \(|2F_2|\) and \(|2F_3|\) as~\(R_{1,1}\). By the first part of the statement, \(R = R_{1,1}\) is in a half-fiber of \(|2F_2|\) or \(|2F_3|\). Hence, \(R_{1,j}\) is in a half-fiber of \(|2F_2|\) or \(|2F_3|\). A similar argument applies to the other \(R_{i,j}\).
\end{proof}

\begin{lemma} \label{lem: Giacomo's.lemma}
Let \((F_1,F_1+R_1,F_2,F_2+R_2)\) be a \(2\)-degenerate \(4\)-sequence. If \(R_2\) is a component of a half-fiber of~\(|2F_1|\), then \(R_1\) is not a component of a half-fiber of~\(|2F_2|\).
\end{lemma}
\begin{proof}
Suppose without loss of generality that \(R_2\) is a component of~\(F_1\) and that \(R_1\) is a component of~\(F_2\). Then, both half-fibers \(F_1\) and \(F_2\) are reducible. 
Recall that \(F_1\) and \(F_2\) do not share components by \autoref{lem: F1.F2.no.common.components}.
The half-fiber \(F_1\) has exactly one component which is a special bisection of~\(|2F_2|\), and that component must be \(R_2\), since \(R_2.F_2 = 1\). Analogously, \(R_1\) is the unique component of~\(F_2\) which is a special bisection of~\(|2F_1|\). This implies that \(R_1\) and \(R_2\) must meet, which contradicts the fact that \(R_1.R_2 = 0\).
\end{proof}

We can now prove the theorem.

\begin{proof}[Proof of \autoref{thm:non-extendable.non-special}]
Any Enriques surface of type~\(2\tilde{D}_4\) admits a non-special non-extendable \(3\)-sequence (see \autoref{example:2D4}), so it suffices to prove the first statement.

Fix a non-special non-extendable \(3\)-sequence \((F_1, F_2, F_3)\). By \autoref{thm: 10sequence}, there exists a \(3\)-degenerate \(10\)-sequence as in \eqref{eq:10-sequence}. Without loss of generality, we may assume \(r_1 \geq 3\). By \autoref{lem: specialtycriterion}, we may assume that \(R_{1,1}\) is contained in \(F_3\).

Consider the vector \(e_{2,3} \in \Num(X)\) given in \autoref{lem:Cossec's.lemma}. By Riemann--Roch, there exists an effective divisor \(D \in \Pic(X)\), whose numerical class is \(e_{2,3}\), such that \(D.F_i=1\), \(D.R_{1,1} = 1\), \(D.R_{1,3} = -1\) and \(D.R_{i,j}=0\) for all other \(i,j\). 
By \autoref{lem: reducibility}, there exist \(D'\) and \(C\) with \(D \sim D' + C\), such that \(D'\) is nef, \(D^2 = D'^2 = 0\), and \(C\) is a linear combination of \((-2)\)-curves with non-negative coefficients.
Since \(D'.F_i \leq 1\) for all \(i\) and \((F_1,F_2,F_3)\) is non-extendable, \(D' \equiv F_i\) for some \(i \in \{1,2,3\}\). 

Let \(E \coloneqq F_1 + F_2 + F_3 - D\). 
Then, \(E^2=0\) and \(E.F_i = 1\). 
Again by \autoref{lem: reducibility}, we can write \(E \sim E' + C'\) with \(E' \equiv F_j\) for some \(j \in \{1,2,3\}\). 
Notice that \(i \neq j\), since otherwise \(F_1+F_2+F_3-2F_i\) would be numerically equivalent to the effective divisor~\(C+C'\), contradicting the fact that the \(3\)-sequence is non-special. 
Therefore, \(C+C'\) is numerically equivalent to one of the \(F_i\), so it coincides with one of the half-fibers of \(|2F_1|\), \(|2F_2|\) or \(|2F_3|\). Since \(C.R_{1,3}=-1\), \(C\) contains \(R_{1,3}\), so \(C+C' = F_3\). In particular, either \(D'\equiv F_1\) and \(E'\equiv F_2\), or vice versa.

\begin{claim} \label{claim: tentative1}
There exists a \(3\)-degenerate \(10\)-sequence extending \((F_1,F_2,F_3)\) as in \eqref{eq:10-sequence}, such that, after possibly permuting the half-fibers, \(r_3 = 0\) and all \(R_{i,j}\) are contained in \(F_3\).
\end{claim}
\begin{proof}
If \(r_2 = r_3 = 0\), the desired property is already satisfied.
Hence, assume that \(r_2>0\) or \(r_3>0\).

Suppose first \(D'\equiv F_1\) and \(E'\equiv F_2\). 
Since \(C'.F_2 = 1\), \(C'\) contains a special bisection \(R\) of \(|2F_2|\). 
Now, notice that \(R_{1,1}\) is not a component of \(C'\), since it is simple in \(F_3\) and already contained in \(C\). 
Moreover, \(C'.R_{1,1}=(E-F_2).R_{1,1} = 0\), so \(R_{1,1}\) is disjoint from \(C'\). 
In particular, \(R_{1,1}.R=0\), and any extension of \((F_1,F_1+R_{1,1},F_2,F_2+R,F_3)\) satisfies the desired property by \autoref{lem: Giacomo's.lemma}.

Suppose instead that \(D'\equiv F_2\) and \(E'\equiv F_1\). If \(r_2>0\), then \(C.R_{2,1}=-1\), so \(R_{2,1}\) is contained in \(F_3\). 
In particular, \(r_3=0\) by \autoref{lem: Giacomo's.lemma} and we are done.
Hence, we may assume \(r_3>0\), and that \(R_{3,1}\) is in \(F_2\) by \autoref{lem: Giacomo's.lemma}. 
Since \(C'.R_{1,1}=(E-F_1).R_{1,1}=-1\), \(C'\) contains \(R_{1,1}\). 
On the other hand, \(C.F_2 = (D-F_2).F_2 =1\), so \(C\) contains a special bisection \(R\) of \(|2F_2|\). 
However, \(C.R_{3,1}=0\) and \(C\) does not contain \(R_{3,1}\), so \(R_{3,1}\) is disjoint from \(C\). 
In particular, \(R.R_{3,1}=0\). Then, \((F_2,F_2+R,F_3,F_3+R_{3,1})\) is a \(2\)-degenerate \(4\)-sequence, whose existence contradicts \autoref{lem: Giacomo's.lemma}.
\end{proof}

We replace our original \(3\)-degenerate \(10\)-sequence by one satisfying the properties of \autoref{claim: tentative1}. 
In particular, \(r_1+r_2 = 7\), and we can suppose that \(r_1 \geq 4\).
Denote by \(G_1\in |2F_1|\) and \(G_2 \in |2F_2|\) the fibers containing \(R_{1,3}\). 
By~\autoref{lem: F1.F2.no.common.components}, both \(G_1\) and \(G_2\) are simple fibers.

\begin{claim}
\(F_3\) is a half-fiber of type~\(\I_4^*\).
\end{claim}
\begin{proof}
Since \(F_3\) contains two disjoint configurations of type~\(A_{r_{1}}\) and \(A_{r_2}\) with \(r_1 + r_2 = 7\), we deduce that \(F_3\) is of type~\(\I_8\), \(\I_9\), \(\I_4^*\), \(\III^*\) or \(\II^*\).

Since \(C.R_{1,3}=-1\), the curve \(R_{1,3}\) is a component of \(C\). Repeating the argument with the other \(R_{i,j}\), we have that \(C-R_{1,2}-\ldots-R_{1,r_1}\) is effective. However, this divisor intersects \(R_{1,3}\) negatively, so reasoning similarly we obtain that the divisor
\[
    Z \coloneqq C-R_{1,2}-2R_{1,3}-\ldots-2R_{1,r_1-1}-R_{1,r_1}
\]
is effective. 
In particular, \(R_{1,3}\) has multiplicity \(\ge 2\) in \(C\), and, therefore, \(R_{1,3}\) has multiplicity \(\ge 2\) in \(F_3\). So, \(F_3\) is additive. Also, since \(C\) and \(C'\) both intersect one of the half-fibers with multiplicity \(1\) and \(C + C' = F_3\), each of the curves \(C\) and \(C'\) contains a simple component of~\(F_3\), so \(F_3\) cannot be of type~\(\II^*\). It remains to exclude the case that \(F_3\) is of type~\(\III^*\).

Seeking a contradiction, assume that \(F_3\) is of type~\(\III^*\). Denote by \(R\) the only component of \(F_3\) different from the \(R_{i,j}\). The only possible way of fitting the \(R_{i,j}\) in \(F_3\) is if \(r_1=7\) or \(r_1 =5\). The former is impossible, since otherwise \(G_2\) would contain both simple components of \(F_3\), contradicting \autoref{prop: G2exists}. Hence, we may assume that \(r_1 = 5\), \(r_2 = 2\), and that the divisor \(Z = C-R_{1,2}-2R_{1,3}-2R_{1,4}-R_{1,5}\) is effective. We have \(R.R_{1,4} = R.R_{2,2} = 1\) and by writing \(R\) as linear combination of \(F_3\) and the \(R_{i,j}\), we obtain \(D.R =1\).

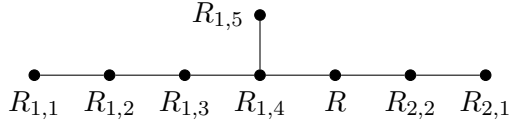
\[
\begin{tikzpicture}[yscale=0.8]
    \node (R1) at (-3, 0) [nodal, label=below:\(R_{1,1}\)] {};
    \node (R2) at ( 3, 0) [nodal, label=below:\(R_{2,1}\)] {};
    \node (R3) at (-2, 0) [nodal, label=below:\(R_{1,2}\)] {};
    \node (R4) at ( 2, 0) [nodal, label=below:\(R_{2,2}\)] {};
    \node (R5) at (-1, 0) [nodal, label=below:\(R_{1,3}\)] {};
    \node (R6) at ( 1, 0) [nodal, label=below:\(R\)] {};
    \node (R7) at ( 0, 0) [nodal, label=below:\(R_{1,4}\)] {};
    \node (R8) at (0, 1) [nodal, label=left:\(R_{1,5}\)] {};
    \draw (R2)--(R4)--(R6);
    \draw (R1)--(R3)--(R5)--(R7)--(R6) (R7)--(R8);
\end{tikzpicture}
\]

Since \(e_{2,3}\) is not contained in the lattice spanned by the \(3\)-degenerate \(10\)-sequence by \autoref{cor: divisibility}, neither is \(Z\), so \(Z\) must contain \(R\).
Thus, we can recursively remove the following curves from \(Z\) while staying effective: \(R,R_{2,2},R_{2,1},R_{2,2},R,R_{1,4},R_{1,3},R_{1,2},\) and \(R_{1,1}\). But then, \(C\) contains both simple components of \(F_3\), which is absurd.
\end{proof}

By \autoref{prop: G2exists}, both \(G_1\) and \(G_2\) contain a configuration of type~\(D_8\), so they are of type~\(\II^*\) or \(\I_4^*\). In the former case, the component of \(G_i\) not contained in \(F_3\) would be a special bisection of \(|2F_3|\) and \(X\) would be of type~\(\tilde{D}_8\), which is absurd. Hence, both \(G_1\) and \(G_2\) are of type~\(\I_4^*\). If the extra components \(R_i\) of \(G_i\) lie on the same side of \(F_3\), then we get the following graph: 
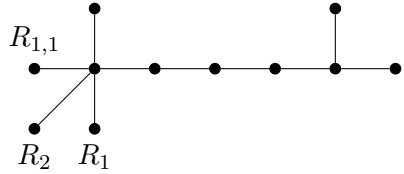
\[
\begin{tikzpicture}[scale=0.8]
    \node (F1) at (-2,-1) [nodal, label=below:\(R_{1}\)] {};
    \node (F2) at (-3,-1) [nodal, label=below:\(R_{2}\)] {};
    \node (R1) at (-3, 0) [nodal, label=above:\(R_{1,1}\)] {};
    \node (R2) at ( 3, 0) [nodal] {};
    \node (R3) at (-2, 0) [nodal] {};
    \node (R4) at ( 2, 0) [nodal] {};
    \node (R5) at (-1, 0) [nodal] {};
    \node (R6) at ( 1, 0) [nodal] {};
    \node (R7) at ( 0, 0) [nodal] {};
    \node (R9) at (-2, 1) [nodal] {};
    \node (R10) at (2, 1) [nodal] {};
       \draw (F1)--(R3) (F2)--(R3);
    \draw (R2)--(R4)--(R6) (R4)--(R10);
    \draw (R1)--(R3)--(R5)--(R7)--(R6) (R3)--(R9);
\end{tikzpicture}
\]
There is a half-fiber of type~\(\I_0^*\) and a configuration of type~\(D_5\) disjoint from this half-fiber. This contradicts \autoref{lem: combinatorial.0}.
Therefore, \(R_1\) and \(R_2\) lie on opposite sides of \(F_3\) and we get the dual graph of type~\(2\tilde{D}_4\). This finishes the proof.
\end{proof}

\section{Projective models of Enriques surfaces} \label{sec: projective.models}

In this section, we describe an application of \autoref{cor:non-degeneracy.4} to the theory of projective models of Enriques surfaces. In particular, we show that every Enriques surface in characteristic different from \(2\) arises via Enriques' original construction. For this, we note first that every Enriques surface of non-degeneracy \(4\) admits a non-special \(3\)-sequence.

\begin{proposition} \label{prop: atmostoneisspecial}
If \((F_1,F_2,F_3,F_4)\) is a \(4\)-sequence, then at most one of the \(3\)-sequences \((F_i,F_j,F_k)\) with \(\{i,j,k\} \subseteq \{1,2,3,4\}\) is special.
\end{proposition}
\begin{proof}
Seeking a contradiction, assume that two of the \(3\)-sequences are special. By \autoref{prop: propertiesof3sequences}, we may assume without loss of generality that there exist \(S_3 \in |F_1 + F_2 - F_3|\) and \(S_4 \in |F_1 + F_2 - F_4|\). Then,
\[
S_3.S_4 = (F_1 + F_2 - F_3).(F_1 + F_2 - F_4) = -1.
\]
Hence, \(S_3\) and \(S_4\) share a component. This is absurd, because \(F_3 + S_3, F_4 + S_4 \in |F_1 + F_2|\), \(F_3 + S_3 \neq F_4 + S_4\), and \(|F_1 + F_2|\) has no base components by \cite[Proposition~2.6.1]{CossecDolgachevLiedtke}.
\end{proof}

\begin{corollary} \label{cor: whendoesthereexistanon-special3sequence?}
For an Enriques surface \(X\), the following are equivalent:
\begin{enumerate}
    \item \(X\) is not of type~\(\tilde{E}_8,\tilde{D}_8,\tilde{E}_7,\) or \(\tilde{E}_7^{(2)}\).
    \item \(X\) admits a non-special \(3\)-sequence.
\end{enumerate}
\end{corollary}
\begin{proof}
If \(X\) admits a non-special \(3\)-sequence, then \(X\) is not of type~\(\tilde{E}_8,\tilde{D}_8,\tilde{E}_7,\) or \(\tilde{E}_7^{(2)}\), since the first three types of surfaces satisfy \(\max \nd(X) \leq 2\) and the unique \(3\)-sequence on the latter type of surface is special, as we have seen in \autoref{example:E7(2)}.

Conversely, if \(X\) is not of one of the excluded types, then \autoref{thm: non-degeneracy.1}, \autoref{thm: non-degeneracy.2}, and \autoref{cor:non-degeneracy.4} show that \(X\) admits a \(4\)-sequence. By \autoref{prop: atmostoneisspecial}, this implies that \(X\) admits a non-special \(3\)-sequence.
\end{proof}

By \cite[Theorem~3.5.1]{CossecDolgachevLiedtke}, a non-special \(3\)-sequence \((F_1,F_2,F_3)\) on a classical Enriques surface \(X\) induces, via the linear system \(|F_1 + F_2 + F_3|\), a realization of~\(X\) as the minimal desingularization of an Enriques sextic~\eqref{eq: Enriques.sextic}.
Thus, the following is an immediate consequence of \autoref{cor: whendoesthereexistanon-special3sequence?}.

\begin{theorem} \label{thm: sexticmodel}
Any classical Enriques surface which is not of type~\(\tilde{E}_8,\tilde{D}_8,\tilde{E}_7\) or~\(\tilde{E}_7^{(2)}\) is the minimal resolution of an Enriques sextic \eqref{eq: Enriques.sextic}.
In particular, if \(p \neq 2\), then \emph{every} Enriques surface arises via this construction.
\end{theorem}

Following an observation of Castelnuovo, we note that the Cremona transformation
\[
    [x_0:x_1:x_2:x_3] \mapsto [x_2x_3:x_0x_1:x_0x_2:x_0x_3]
\]
maps an Enriques sextic~\eqref{eq: Enriques.sextic} into a surface of the form
\begin{equation*}
    x_0^4x_1^2x_2^4x_3^2 + x_0^4x_1^2x_2^2x_3^4 + x_0^4x_2^4x_3^4 + x_0^6x_1^2x_2^2x_3^2 + x_0^3x_1x_2^2x_3^2Q'.
\end{equation*} 
Upon dividing by \(x_0^3x_2^2x_3^2\), we obtain the Castelnuovo quintic~\eqref{eq: Castelnuovo.quintic},
which has elliptic singularities at the four vertices of the coordinate tetrahedron (see \cite[Example~1.6.4]{CossecDolgachevLiedtke}).
In terms of linear systems, and after suitably permuting the coordinates, the above quintic model corresponds to the linear subsystem of \(|2F_1 + F_2 + F_3|\) spanned by the four divisors \(2F_1' + F_2' + F_3'\), \(2F_1 + F_2' + F_3'\), \(F_1 + F_1' + F_2 + F_3'\), and \(F_1 + F_1' + F_2' + F_3\), where \(F_i'\) denotes the second half-fiber of~\(|2F_i|\).

\begin{theorem} \label{thm: quinticmodel}
Any classical Enriques surface which is not of type~\(\tilde{E}_8,\tilde{D}_8,\tilde{E}_7\) or~\(\tilde{E}_7^{(2)}\) is birational to a normal Castelnuovo quintic~\eqref{eq: Castelnuovo.quintic}.
In particular, if \(p \neq 2\), then \emph{every} Enriques surface arises via this construction.
\end{theorem}

\begin{remark}
Quintic models of Enriques surfaces in characteristic \(p \neq 2\) were previously studied by Castelnuovo~\cite{Castelnuovo}, Kim~\cite{Kim}, Stagnaro~\cite{Stagnaro} and Umezu~\cite{Umezu}. Applying \cite[Corollary 2.3]{Umezu} together with our \autoref{cor:non-degeneracy.4} yields \autoref{thm: quinticmodel} in these characteristics, since Stagnaro's first construction of quintic Enriques surfaces \cite{Stagnaro} coincides with the one obtained via Castelnuovo's transformation. 
\end{remark}

\begin{remark}
It was known already to Enriques \cite{Enriques:bigenere} that every complex Enriques surface is birational to a double cover of~\(\mathbb{P}^2\) branched over (a degeneration of) an \emph{Enriques octic}, which is the union of two lines and a singular sextic, such that the sextic has two tacnodes whose tacnodal tangents coincide with the lines and a node at the intersection of the two lines. 
Such a double plane model can be obtained by choosing a \(2\)-sequence \((F_1,F_2)\) on~\(X\) and composing the map from \(X\) to a \(4\)-nodal quartic del Pezzo induced by the complete linear system \(|2F_1 + 2F_2|\) with the projection from a line through two of the nodes. 
The resulting rational map to \(\mathbb{P}^2\) corresponds to the linear system \(|2F_1 + F_2|\).

Our results also show that we can assume these Enriques octics to satisfy a certain non-degeneracy condition, in every characteristic different from \(2\). 
Choose a non-special \(3\)-sequence \((F_1,F_2,F_3)\) on~\(X\). The composition of the map from \(X\) to the Castelnuovo quintic with the projection from one of the two elliptic singularities of multiplicity \(3\) is given, after a suitable permutation of coordinates, by the linear system \(|2F_1 + F_2|\), so this double plane model coincides with the one described in the previous paragraph. In terms of equations, we can write the quintic as 
\[
    x_3^2 C_1 + x_0x_1x_3 Q'' + x_0x_1 C_2,
\]
where the \(C_i\) are equations of cubics passing through \([0:0:1]\), tangent to \(\{x_0 = 0\}\) at \([0:1:0]\), and tangent to \(\{x_1 = 0\}\) at \([1:0:0]\), and \(Q''\) is the equation of a conic passing through \([1:0:0]\) and \([0:1:0]\). The corresponding Enriques octic is given by \(x_0x_1(x_0x_1Q''^2 - 4C_1C_2)\). In particular, we get the existence of two cubics and a conic that are in special position with respect to the Enriques octic. The two cubics given by \(C_1\) and \(C_2\) are the images of \(F_3\) and \(F_3'\). Hence, our results show that every Enriques surface in characteristic different from \(2\) is birational to a double cover of \(\mathbb{P}^2\) branched over an Enriques octic that admits two cubics \(C_1\) and \(C_2\) as above.
\end{remark}

\bibliographystyle{amsplain}
\bibliography{Enriques}
\end{document}